\documentclass{amsart}

\usepackage[T1]{fontenc}
\usepackage[utf8]{inputenc}
\usepackage{amsmath,amsfonts,amsthm,amssymb,amsxtra,bbm}
\usepackage{fourier,setspace,graphicx,color,pdflscape}
\usepackage{textcomp}
\usepackage{MnSymbol}
\usepackage{color}
\usepackage[colorlinks=true]{hyperref}
\hypersetup{urlcolor=blue, linkcolor=blue, citecolor=red, anchorcolor=blue]}
\usepackage{mathrsfs}
\usepackage{mathtools}
\usepackage{setspace}
\usepackage{lineno,hyperref}
\newtheorem{theorem}{Theorem}
\newtheorem{proposition}[theorem]{Proposition}
\newtheorem{lemma}[theorem]{Lemma}
\newtheorem{corollary}[theorem]{Corollary}

\newcommand{\R}{{\mathbb R}}
\newcommand{\N}{{\mathbb N}}

\renewcommand{\S}{{\mathbb S}}
\newcommand{\be}[1]{\begin{equation}\label{#1}}

\newcommand{\ee}{\end{equation}}
\renewcommand{\(}{\left(}
\renewcommand{\)}{\right)}

\newcommand{\isd}[1]{\int_{\S^d}{#1}\,d\mu}
\newcommand{\nrm}[2]{\left\|#1\right\|_{\mathrm L^{#2}(\R^d)}}
\newcommand{\nrms}[2]{\left\|#1\right\|_{\mathrm L^{#2}(\S^d)}}
\newcommand{\irdg}[1]{\int_{\R^d}{#1}\,d\sigma}
\newcommand{\nrmG}[2]{\left\|#1\right\|_{\mathrm L^{#2}(\R^d,d\sigma)}}

\definecolor{darkgreen}{rgb}{0,0.4,0}

\makeatletter
\@namedef{subjclassname@2020}{%
 \textup{2020} Mathematics Subject Classification}
\makeatother
\newcommand{\msc}[1]{\href{https://mathscinet.ams.org/mathscinet/search/mscbrowse.html?sk=default&sk=#1&submit=Search}{#1}}

\newcommand{\step}[1]{\par\smallskip\noindent{$\bullet$ \emph{#1.}}}

\usepackage{etoolbox}
\newcounter{taggedeq}
\pretocmd{\equation}{\stepcounter{taggedeq}}{}{}

\numberwithin{equation}{section}

\begin{document}

\title[Stability of subcritical interpolation inequalities on the sphere]{Logarithmic Sobolev and interpolation inequalities on the sphere: constructive stability results}

\author[G.~Brigati]{Giovanni Brigati}
\address[G.~Brigati]{CEREMADE (CNRS UMR n$^\circ$~7534), PSL University, Universit\'e Paris-Dauphine,\newline Place de Lattre de Tassigny, 75775 Paris 16, France}
\email{brigati@ceremade.dauphine.fr}
\author[J.~Dolbeault]{Jean Dolbeault}
\address[J.~Dolbeault]{CEREMADE (CNRS UMR n$^\circ$~7534), PSL University, Universit\'e Paris-Dauphine,\newline Place de Lattre de Tassigny, 75775 Paris 16, France}
\email{dolbeaul@ceremade.dauphine.fr}
\author[N.~Simonov]{Nikita Simonov}
\address[N.~Simonov]{LJLL (CNRS UMR n$^\circ$~7598) Sorbonne Universit\'e,\newline 4 place Jussieu, 75005 Paris, France}
\email{nikita.simonov@sorbonne-universite.fr}

\begin{abstract} We consider Gagliardo--Nirenberg inequalities on the sphere which interpolate between the Poincar\'e inequality and the Sobolev inequality, and include the logarithmic Sobolev inequality as a special  case. We establish explicit stability results in the subcritical regime using spectral decomposition techniques, and entropy and \emph{carr\'e du champ} methods applied to nonlinear diffusion flows.
\end{abstract}

\maketitle
\thispagestyle{empty}
\vspace*{-0.6cm}

%%%%%%%%%%%%%%%%%%%%%%%%%%%%%%%%%%%%%%%%%%%%%%%%%%%%%%%%%%%%%%%%%%%%%%%%%%
%%%%%%%%%%%%%%%%%%%%%%%%%%%%%%%%%%%%%%%%%%%%%%%%%%%%%%%%%%%%%%%%%%%%%%%%%%
\section{Introduction and main results}\label{Sec:Intro}

Functional inequalities are essential in many areas of mathematics. The knowledge of optimal constants, or at least good estimates of them, is crucial for various applications. Whether optimality cases are achieved is a standard issue in analysis. The next natural question is to understand how the deficit, say the difference of the two sides of the functional inequality, measures the distance to the set of optimal functions. Such a question has been actively studied in critical Sobolev inequalities, but much less in subcritical interpolation inequalities. In the case of the sphere, a global stability result based on Bianchi--Egnell-type methods was recently obtained for a family of Gagliardo--Nirenberg inequalities by Frank in~\cite{Frank_2022}, with the striking observation that only the power $4$ of a natural distance is controlled by the deficit. Here we give a more detailed picture, which includes the logarithmic Sobolev inequality, and provide explicit estimates.

\medskip On the sphere $\S^d$ with $d\ge1$, the \emph{logarithmic Sobolev inequality} can be written~as
\be{LogSobSphere}\tag{LS}
\isd{|\nabla F|^2}\ge\frac d2\isd{F^2\,\log\(\frac{F^2}{\nrms F2^2}\)}\quad\forall\,F\in\mathrm H^1(\S^d,d\mu)\,,
\ee
where $d\mu$ denotes the uniform probability measure. The equality case is achieved by constant functions and $d/2$ is the optimal constant as shown by taking the test functions $F_\varepsilon(x)=1+\varepsilon\,x\cdot\nu$, for some arbitrary $\nu\in\S^d$, in the limit as $\varepsilon\to0$. Our first result is an improved inequality under an orthogonality constraint, which improves upon~\cite[Proposition~5.4]{1504}.
%-------------------------------------------------------------------------
\begin{theorem}\label{Thm:sphere-LSI-improved} Let $d\ge1$. For any $F\in\mathrm H^1(\S^d,d\mu)$ such that
\be{orth}
\isd{x\,F}=0\,,
\ee
we have
\be{Ineq:LSI-improved}
\isd{|\nabla F|^2}-\frac d2\isd{F^2\,\log\(\frac{F^2}{\nrms F2^2}\)}\ge\mathscr C_d\isd{|\nabla F|^2}\,,
\ee
with $\mathscr C_d=\frac2{d+2}$. 
\end{theorem}
%-------------------------------------------------------------------------
\noindent Since equality in~\eqref{LogSobSphere} is achieved if and only if $F$ is a constant function, the right-hand side in~\eqref{Ineq:LSI-improved} is an estimate of the distance to the set of optimal functions under the constraint $\isd{x\,F}=0$. Alternatively, Theorem~\ref{Thm:sphere-LSI-improved} amounts to the \emph{improved logarithmic Sobolev inequality}
\[
\isd{|\nabla F|^2}\ge\frac{d+2}2\isd{F^2\,\log\(\frac{F^2}{\nrms F2^2}\)}\quad\forall\,F\in\mathrm H^1(\S^d,d\mu)\;\mbox{s.t.}\;\isd{x\,F}=0\,.
\]
Without condition~\eqref{orth}, there is no such inequality as~\eqref{Ineq:LSI-improved}. With $F_\varepsilon(x)=1+\varepsilon\,x\cdot\nu$ as above, as $\varepsilon\to0$ one can indeed check that
\[
\nrms{\nabla F_\varepsilon}2^2-\frac d2\isd{F_\varepsilon^2\,\log\(\frac{F_\varepsilon^2}{\nrms{F_\varepsilon}2^2}\)}=O(\varepsilon^4)=O\(\nrms{\nabla F_\varepsilon}2^4\)\,.
\]
In the absence of an additional constraint, like~\eqref{orth}, such behaviour is in fact optimal. The following estimate arises from the \emph{carr\'e du champ} method.
%-------------------------------------------------------------------------
\begin{proposition}\label{Prop:LSI-BE} Let $d\ge1$, $\gamma=1/3$ if $d=1$ and $\gamma=(4\,d-1)\,(d-1)^2/(d+2)^2$ if $d\ge2$. Then, for any $F\in\mathrm H^1(\S^d,d\mu)$ we have
\[
\isd{|\nabla F|^2}-\frac d2\isd{F^2\,\log\(\frac{F^2}{\nrms F2^2}\)}\ge\frac12\,\frac{\gamma\,\nrms{\nabla F}2^4}{\gamma\,\nrms{\nabla F}2^2+d\,\nrms F2^2}\,.
\]
\end{proposition}
%-------------------------------------------------------------------------
\noindent With $\nrms F2^2=1$, notice that the deficit can be estimated from below by 
\[
\isd{|\nabla F|^2}-\frac d2\isd{F^2\,\log\(F^2\)}\ge\frac\gamma{2\,d}\,\nrms{\nabla F}2^4+o\(\nrms{\nabla F}2^4\)
\]
if $\nrms{\nabla F}2^2$ is small enough.

Let $\Pi_1 F$ denote the orthogonal projection of a function $F\in\mathrm L^2(\S^d)$ on the spherical harmonics corresponding to the first positive eigenvalue of $-\,\Delta$ on $\S^d$, \emph{i.e.},
\[
\Pi_1 F(x)=(d+1)\,x\cdot\int_{\S^d}y\,F(y)\,d\mu(y)\quad\forall\,x\in\S^d\,.
\]
Our main stability result for the logarithmic Sobolev inequality combines the results of Theorem~\ref{Thm:sphere-LSI-improved} and Proposition~\ref{Prop:LSI-BE} as follows.
%-------------------------------------------------------------------------
\begin{theorem}\label{Thm:sphere-LSI-stability} Let $d\ge1$. For any $F\in\mathrm H^1(\S^d,d\mu)$, we have
\begin{multline*}
\isd{|\nabla F|^2}-\frac d2\isd{F^2\,\log\(\frac{F^2}{\nrms F2^2}\)}\\
\ge\mathscr S_d\(\frac{\nrms{\nabla\Pi_1 F}2^4}{\nrms{\nabla F}2^2+\frac d2\,\nrms F2^2}+\nrms{\nabla(\mathrm{Id}-\Pi_1)\,F}2^2\)
\end{multline*}
for some stability constant $\mathscr S_d>0$.
\end{theorem}
%-------------------------------------------------------------------------
\noindent An explicit estimate of $\mathscr S_d$ is given in Section~\ref{Sec:Global}.

\medskip We also consider the subcritical \emph{Gagliardo--Nirenberg inequalities}
\be{GNS}\tag{GN}
\isd{|\nabla F|^2}\ge\frac d{p-2}\(\nrms Fp^2-\nrms F2^2\)\quad\forall\,F\in\mathrm H^1(\S^d,d\mu)\,,
\ee
for any $p\in[1,2)\cup(2,2^*)$. Here, $d\mu$ again denotes the uniform probability measure on~$\S^d$, the critical Sobolev exponent is $2^*:=2\,d/(d-2)$ if $d\ge3$ and we adopt the convention that $2^*=+\infty$ if $d=1$ or $d=2$. Inequality~\eqref{GNS} with $p=1$ is equivalent to the Poincar\'e inequality. If $d\ge3$, inequality~\eqref{GNS} also holds for the critical exponent $p=2^*$ and it is in fact Sobolev's inequality with optimal constant on $\S^d$, but this is out of the scope of our paper which focuses on the subcritical regime $p<2^*$. The logarithmic Sobolev inequality~\eqref{LogSobSphere} is obtained from~\eqref{GNS} by taking the limit as $p\to2$, and the counterpart of the above results for $p\neq2$, in the \emph{subcritical range }$p<2^*$, goes as follows.
%-------------------------------------------------------------------------
\begin{theorem}\label{StabSubcriticalSphereGNS} Assume that $d\ge1$ and $p\in(1,2)\cup(2,2^*)$. For any function $F\in\mathrm H^1(\S^d,d\mu)$ such that the orthogonality condition~\eqref{orth} holds, we have
\be{BVVimproved}
\isd{|\nabla F|^2}-\frac d{p-2}\(\nrms Fp^2-\nrms F2^2\)\ge\mathscr C_{d,p}\isd{|\nabla F|^2}
\ee
with $\mathscr C_{d,p}=\frac{2\,d-p\,(d-2)}{2\,(d+p)}$.
\end{theorem}
%-------------------------------------------------------------------------
Taking $F_\varepsilon(x)=1+\varepsilon\,x\cdot\nu$ as above shows that~\eqref{orth} is needed in Theorem~\ref{StabSubcriticalSphereGNS}. We also have a higher-order estimate of the deficit as a consequence of the \emph{carr\'e du champ} method.
%-------------------------------------------------------------------------
\begin{proposition}\label{Prop:GNSBE} Let $d\ge1$ and $p\in(1,2)\cup(2,2^*)$. There is a convex function $\psi$ on $\R^+$ with $\psi(0)=\psi'(0)=0$ such that, for any $F\in\mathrm H^1(\S^d,d\mu)$, we have
\[
\isd{|\nabla F|^2}-\frac d{p-2}\(\nrms Fp^2-\nrms F2^2\)\ge\nrms Fp^2\,\psi\(\frac{\nrms{\nabla F}2^2}{\nrms Fp^2}\)\,.
\]
\end{proposition}
%-------------------------------------------------------------------------
\noindent An explicit expression for $\psi$ will be given in Section~\ref{Sec:Improved-Carre}. The two results of Theorem~\ref{StabSubcriticalSphereGNS} and Proposition~\ref {Prop:GNSBE} can be combined to prove the analogue of Theorem~\ref{Thm:sphere-LSI-stability} for $p\neq2$, with an explicit constant: see Section~\ref{Sec:Global}.
%-------------------------------------------------------------------------
\begin{theorem}\label{Thm:sphere-GNS-stability} Let $d\ge1$ and $p\in(1,2)\cup(2,2^*)$. For any $F\in\mathrm H^1(\S^d,d\mu)$, we have
\begin{multline*}
\isd{|\nabla F|^2}-\frac d{p-2}\(\nrms Fp^2-\nrms F2^2\)\\
\ge\mathscr S_{d,p}\(\frac{\nrms{\nabla\Pi_1 F}2^4}{\nrms{\nabla F}2^2+\nrms F2^2}+\nrms{\nabla(\mathrm{Id}-\Pi_1)\,F}2^2\)
\end{multline*}
for some explicit stability constant $\mathscr S_{d,p}>0$.
\end{theorem}
%-------------------------------------------------------------------------

Let us give a brief account of the literature. In this paper, we address the distinction between \emph{improved inequalities} (inequalities with improved constants under orthogonality constraints) and \emph{quantitative stability} (as a measure of a distance to the set of optimal functions). There are many adjacent directions of research like, for instance, stability in weaker norms (see for instance~\cite{MR3227280,https://doi.org/10.48550/arxiv.1404.1028} for Sobolev's inequality) or notions of stability with no explicit notion of distance. To our knowledge, not so much has been done in subcritical interpolation inequalities (see~\cite{BDNS,Frank_2022} and some references therein), except for the logarithmic Sobolev inequality, for which we refer to~\cite{MR3493423,MR3567822} and~\cite{MR3271181,MR3666798,MR4305006,MR4455233}.

The \emph{Gagliardo--Nirenberg inequalities} \eqref{GNS} on the sphere have been established with optimal constant for any $p\in(2,2^*)$ in~\cite[Corollary~6.1]{bidaut1991nonlinear} and in~\cite{MR1230930}. In dimension $d=2$, Onofri's inequality is obtained from~\eqref{GNS} in the limit as $p\to2^*=+\infty$: see~\cite{MR1230930,MR1143664}. With $p\in[1,2)$ or $p>2$ but not too large (if $d\ge2$), inequality~\eqref{GNS} was known earlier from~\cite{bakry1985diffusions}. A Markovian point of view is presented in~\cite{MR3155209}, with many more references therein on related questions. On Euclidean space, similar inequalities go back to~\cite{zbMATH03035784,MR0102740,MR0109940}. The \emph{logarithmic Sobolev inequality} \eqref{LogSobSphere} is a well known-limit case as $p\to2$ and can be considered in a common framework with~\eqref{GNS}. Whenever possible, we shall adopt this point of view. For an overview of early results on the sphere, we refer to~\cite[Section 6,~(iv)]{MR1292277}. The literature on~\eqref{LogSobSphere} on the circle and on the sphere can be traced back at least to~\cite{Weissler_1980},~\cite[Theorem~1, page 268]{Mueller_1982} with computations based on the ultraspherical operator, and~\cite{MR620582} for a more variational approach. The inequality with optimal constant is stated in~\cite[inequality (13) page 195]{bakry1985diffusions} as a consequence of the \emph{carr\'e du champ} method. Also see~\cite{bakry1985inegalites} and~\cite[page 342]{carlen1986application} for related results and~\cite{MR2381156,MR3229793,1504} for a PDE approach based on entropy estimates and the \emph{carr\'e du champ} method. After Schwarz foliated symmetrization, the problem is reduced to a simpler family of interpolation inequalities involving only the ultraspherical operator.

The interest for stability issues was raised by~\cite{MR790771} and the stability result of Bianchi and Egnell in~\cite{MR1124290}, on Euclidean space. Over the years, various approaches have been developed, based on compactness methods and contradiction arguments as in~\cite{MR1124290,MR3179693}, spectral analysis and orthogonality conditions as in~\cite[Proposition~5.4]{1504} and~\cite{MR3878729}, or entropy methods and improved inequalities as in~\cite{MR2152502,MR3103175,DEKL,1504,Dolbeault_2020b}. For spectral methods, a fruitful strategy relies on the Funk--Hecke formula, which is behind \eqref{interpolation}, and the approach of~\cite{Lieb-83,MR1230930}, which applies to the stability result for fractional interpolation inequalities of~\cite{MR3179693} and~\cite[Corollary~2.3]{Dolbeault_2016}. This is the method we use in Section~\ref{Sec:Improved-Constraints}. Stability issues for~\eqref{GNS} have recently been discussed in~\cite{Frank_2022} with methods of Bianchi--Egnell-type, with the drawback that no estimate of the stability constant is known. This drawback can be cured by a \emph{carr\'e du champ} method as we shall see in Section~\ref{Sec:Improved-Carre}. Without entering into details, let us mention some recent progress on stability in~\cite{BDNS,https://doi.org/10.48550/arxiv.2209.08651,https://doi.org/10.48550/arxiv.2210.08482,https://doi.org/10.48550/arxiv.2210.06727} for related critical inequalities.

\medskip This paper is organized as follows. Section~\ref{Sec:Improved-Constraints} is devoted to the proof by spectral methods of Theorem~\ref{StabSubcriticalSphere} (see below), which is an extension of Theorems~\ref{Thm:sphere-LSI-improved} and~\ref{StabSubcriticalSphereGNS}: under orthogonality constraints, these results are reduced to estimates of improved constants in inequalities~\eqref{LogSobSphere} and~\eqref{GNS}, with various refinements based on a decomposition in spherical harmonics. An explicit stability result without constraints corresponding to Propositions~\ref{Prop:LSI-BE} and~\ref{Prop:GNSBE} is proved in Section~\ref{Sec:Improved-Carre}. The proofs of Theorems~\ref{Thm:sphere-LSI-stability} and~\ref{Thm:sphere-GNS-stability}, in Section~\ref{Sec:Global}, is based on the spectral decomposition method developed by Frank in~\cite{Frank_2022}. We collect the previous estimates (with and without orthogonality constraints) in global results, with explicit constants. Various additional results are stated in two appendices: the extension of the method to interpolation inequalities for the Gaussian measure on Euclidean space and a discussion of its limitations in Appendix~\ref{AppendixB}, the details of the computations of the \emph{carr\'e du champ} method on the sphere and its application in order to establish improved functional inequalities in Appendix~\ref{Appendix:BE-Sphere}.

%%%%%%%%%%%%%%%%%%%%%%%%%%%%%%%%%%%%%%%%%%%%%%%%%%%%%%%%%%%%%%%%%%%%%%%%%%
%%%%%%%%%%%%%%%%%%%%%%%%%%%%%%%%%%%%%%%%%%%%%%%%%%%%%%%%%%%%%%%%%%%%%%%%%%
\section{Improvements under orthogonality constraints}\label{Sec:Improved-Constraints}

In this section, we prove Theorems~\ref{Thm:sphere-LSI-improved} and~\ref{StabSubcriticalSphereGNS} in the slightly more general framework of Theorem~\ref{StabSubcriticalSphere} below. Let us consider the generalized entropy functionals
\[
\mathcal E_2[F]:=\frac12\isd{F^2\,\log\(\frac{F^2}{\nrms F2^2}\)}\quad\mbox{and}\quad\mathcal E_p[F]:=\frac{\nrms Fp^2-\nrms F2^2}{p-2}\quad\kern -6pt\mbox{if}\quad\kern -6pt p\neq2\,.
\]
With this notation, we can rephrase~\eqref{LogSobSphere} and~\eqref{GNS} as
\[
\isd{|\nabla F|^2}\ge d\,\mathcal E_p[F]\quad\forall\,F\in\mathrm H^1(\S^d,d\mu)\,,
\]
for any $p\in[1,2^*)$. The optimality case is achieved by considering the test function $F_\varepsilon=1+\varepsilon\,\varphi_1$ in the limit as $\varepsilon\to0$, where $\varphi_1$ is an eigenfunction of the Laplace--Beltrami operator such that $-\,\Delta\varphi_1=d\,\varphi_1$, for instance $\varphi_1(x)=x\cdot\nu$ for some $\nu\in\S^d$ as in Section~\ref{Sec:Intro}.

Let us consider the decomposition into spherical harmonics of $\mathrm L^2(\S^d,d\mu)$,
\[
\mathrm L^2(\S^d,d\mu)=\bigoplus_{\ell=0}^\infty\mathcal H_\ell\,,
\]
where $\mathcal H_\ell$ is the subspace of spherical harmonics of degree $\ell\ge0$. See for instance~\cite{MR0199449,MR0304972,MR0282313,MR1641900}. For any integer $k\ge1$, let us define $\Pi_k$ as the orthogonal projection with respect to $\mathrm L^2(\S^d,d\mu)$ onto $\bigoplus_{\ell=1}^k\mathcal H_\ell$. The following statement extends Theorems~\ref{Thm:sphere-LSI-improved} and~\ref{StabSubcriticalSphereGNS}.
%-------------------------------------------------------------------------
\begin{theorem}\label{StabSubcriticalSphere} Assume that $d\ge1$, $p\in(1,2^*)$ and let $k\ge1$ be an integer. For any function $F\in\mathrm H^1(\S^d,d\mu)$, we have
\be{BVVimproved22}
\isd{|\nabla F|^2}-d\,\mathcal E_p[F]\ge\mathscr C_{d,p,k}\isd{\big|\nabla(\mathrm{Id}-\Pi_k)\, F\big|^2}
\ee
for some explicit constant $\mathscr C_{d,p,k}\in(0,1)$ such that $\mathscr C_{d,p,k}\le\mathscr C_{d,p,1}=\frac{2\,d-p\,(d-2)}{2\,(d+p)}$.
\end{theorem}
%-------------------------------------------------------------------------
\noindent The expression for $\mathscr C_{d,p,k}$ is given below in the proof. inequality~\eqref{BVVimproved22} can be seen as an improvement of~\eqref{LogSobSphere} and~\eqref{GNS}, namely
\[
(1-\mathscr C_{d,p,k})\isd{|\nabla F|^2}\ge d\,\mathcal E_p[F]
\]
for any $F\in\mathrm H^1(\S^d,d\mu)$ such that $\Pi_kF=0$. With $k=1$, this establishes~\eqref{Ineq:LSI-improved} and~\eqref{BVVimproved}, thus proving Theorem~\ref{Thm:sphere-LSI-improved} if $p=2$, and Theorem~\ref{StabSubcriticalSphereGNS} if $p\neq2$.

\begin{proof}[Proof of Theorem~\ref{StabSubcriticalSphere}] Let $(F_j)_{j\in\N}$ be the decomposition of $F$ along $\mathcal H_j$ for any $j\in\N$. We learn from~\cite[Ineq.~(19)]{MR1230930} or~\cite[Ineq.~(1.6)]{Dolbeault_2016} that the \emph{subcritical interpolation inequalities}
\be{interpolation}
\mathcal E_p[F]\le\sum_{j=1}^\infty\zeta_j(p)\isd{|F_j|^2}\quad\forall\,F\in\mathrm H^1(\S^d,d\mu)
\ee
hold for any $p\in(1,2)\cup(2,2^*)$ with
\[
\zeta_j(p):=\frac{\gamma_j\big(\tfrac dp\big)-1}{p-2}\quad\mbox{and}\quad\gamma_j(x):=\frac{\Gamma(x)\,\Gamma(j+d-x)}{\Gamma(d-x)\,\Gamma(x+j)}\,.
\]
This result is based on the Funk--Hecke theorem (see for instance~\cite[Section~4]{MR2848628}) and Lieb's ideas in~\cite{Lieb-83}. Notice that $\zeta_j(p)\ge0$ for any $p\in(1,2)\cup(2,2^*)$. According to~\cite[Lemma~2.2]{Dolbeault_2016}, the function $\zeta_j$ is strictly monotone increasing on $(1,\infty)$ for any $j\ge2$ and the limits
\[
\lambda_j=d\,\lim_{p\to2^*}\zeta_j(p)
\]
are the eigenvalues of the Laplace--Beltrami on the sphere, with $\lambda_j=j\,(j+d-1)$. Hence
\[
d\,\mathcal E_p[F]\le\sum_{j=1}^\infty\lambda_j\isd{|F_j|^2}=\isd{|\nabla F|^2}\,,
\]
which is the essence of the proof of~\eqref{GNS} in~\cite{MR1230930} and also the main idea for the proof of the stability result for fractional interpolation inequalities of~\cite[Corollary~2.3]{Dolbeault_2016}. Here we draw some consequences in standard norms for nonfractional operators and identify estimates of the stability constant in the corresponding stability result.

\medskip\noindent$\rhd$ \emph{The case $p\neq2$.} Let $x=d/p\in((d-2)/2,d]$ if $d\ge2$ and $x\in (0,d]$ if $d=1$. We consider
\[
\xi_j(x):=\frac{|\gamma_j(x)-1|}{j\,(j+d-1)}\,\quad\mbox{and}\quad h_j(x)=\frac{j\,(j+d-1)\,(j+d-x)}{(j+1)\,(j+d)\,(j+x)}\,;
\]
notice that $\gamma_j(x)>1$ for $x<d/2$, while $\gamma_j(x)<1$ for $x>d/2$. An elementary computation shows that $0<h_j(x)<1$. Since $\gamma_{j+1}(x)\,\lambda_\kappa=h_j(x)\,\lambda_{j+1}\,\gamma_j(x)$, we obtain
\be{convex.combination}
\xi_{j+1}(x)=h_j(x)\,\xi_j(x)+ (1-h_j(x))\,\xi_j^\star(x)\,,
\ee
where
\[
\xi_j^\star(x):=\frac1{1-h_j(x)}\left|\frac{h_j(x)}{\lambda_j}-\frac1{\lambda_{j+1}}\right|=
\frac{|d-2\,x|}{j\,(j+d)\(2\,x-d+2\)+d\,x}\,.
\]
Notice that $(\xi_j^\star(x))_{j\ge2}$ is a monotone decreasing sequence for any fixed, admissible value of $x$. We start at $j=2$ with the observation that $\xi_2^\star(x)<\xi_2(x)$ if $x$ is admissible. This gives, by using~\eqref{convex.combination}, the following estimate
\[
\xi_3(x)=h_2(x)\,\xi_2(x)+\big(1-h_2(x)\big)\,\xi_2^\star(x)<\xi_2(x)\,.
\]
Using $\xi_3^\star(x)<\xi_2^\star(x)$, we can iterate and conclude by induction that $\xi_j(x)<\xi_2(x)$ for all $j\ge 3$. As a consequence, we obtain
\[
\sup_{j\ge3}\frac{\zeta_j(p)}{j\,(j+d-1)}<\frac{\zeta_2(p)}{2\,(d+1)}=\frac p{2\,(d+p)}<\frac1d\quad\forall\,p\in(1,2)\cup(2,2^*)\,.
\]
We deduce from~\eqref{interpolation} that
\begin{multline*}
\mathcal E_p[F]\le\isd{|F_1|^2}+\frac p{2\,(d+p)}\sum_{j=2}^\infty j\,(j+d-1)\isd{|F_j|^2}\\
=\frac1d\isd{|\nabla F|^2}+\frac{2\,d-p\,(d-2)}{2\,d\,(d+p)}\isd{|\nabla(\mathrm{Id}-\Pi_1)\,F|^2}\,,
\end{multline*}
which proves the result with $k=1$ and gives the expression for $\mathscr C_{d,p,1}$.

Let us consider the case $k>1$. We already know that $\xi_2(x)>\xi_2^\star(x)$. For any $j\ge2$, we deduce from~\eqref{convex.combination} that
\[
\xi_{j+1}(x)-\xi_{j+1}^\star(x)=h_j(x)\,\big(\xi_j-\xi_j^\star(x)\big)+\xi_j^\star(x)-\xi_{j+1}^\star(x)\ge h_j(x)\,\big(\xi_j-\xi_j^\star(x)\big)
\]
because $j\mapsto\xi_j^\star(x)$ is monotone decreasing. By induction, this proves that $\xi_j(x)>\xi_j^\star(x)$ for any $j\ge2$. As a consequence of~\eqref{convex.combination}, $j\mapsto\xi_j(x)$ is also monotone decreasing and
\[
\sup_{j\ge k+2}\frac{\zeta_j(p)}{j\,(j+d-1)}<\frac{\zeta_{k+1}(p)}{(k+1)\,(k+d)}<\frac1d\quad\forall\,p\in(1,2)\cup(2,2^*)\,.
\]
Altogether, for any $k\ge1$, we have
\[
d\,\mathcal E_p[F]\le\isd{|\nabla\Pi_k F|^2}+\frac{d\,\zeta_{k+1}(p)}{(k+1)\,(k+d)}\isd{|\nabla F|^2}\,,
\]
and the stability constant in~\eqref{BVVimproved22} is estimated by
\[
\mathscr C_{d,p,k}=1-\frac{d\,\zeta_{k+1}(p)}{(k+1)\,(k+d)}\,.
\]
In our method, this constant cannot be improved as shown by a test function such that $F_j=0$ for any $j\in\N$ such that $j\neq 0$ and $j\neq k+1$, but this does not prove the optimality of $\mathscr C_{d,p,k}$.

\medskip\noindent$\rhd$ \emph{The case $p=2$.} By taking the limit as $p\to2_+$ in~\eqref{interpolation}, we obtain that
\[
\eta_j:=\frac2d\lim_{p\to2_+}\zeta_j(p)=\,\psi(j+d/2)-\psi(d/2)\,,
\]
where $\psi(z)=\Gamma'(z)/\Gamma(z)$ is the \emph{digamma function}, and
\[
\frac 12 \isd{F^2\,\log\(\frac{F^2}{\nrms F2^2}\)}\le\frac d2\,\sum_{j=1}^\infty\eta_j\isd{|F_j|^2}\quad\forall\,F\in\mathrm H^1(\S^d,d\mu)\,.
\]
{}From $\psi(z+1)=\psi(z)+1/z$ obtained by differentiating the identity $\Gamma(z+1)=z\,\Gamma(z)$ with respect to $z$, we learn that
\[
\eta_{j+1}=\eta_j+\frac2{d+2\,j}\,.
\]
We claim that
\[
\eta_2\le\eta_j\le\frac{2\,\lambda_j}{d\,(d+2)}\quad\forall\,j\ge2
\]
because there is equality for $j=2$ as $\eta_2=\frac{4\,(d+1)}{d\,(d+2)}$ and $\lambda_2=2\,(d+1)$ on the one hand, and
\[
\eta_{j+1}-\eta_j=\frac2{d+2\,j}\le\frac{2\,(d+2\,j)}{d\,(d+2)}=\frac{2\,(\lambda_{j+1}-\lambda_j)}{d\,(d+2)}
\]
on the other hand, so that the result follows by induction.

Using $\lambda_{j+1}=\lambda_j+(d+2\,j)=\lambda_j+2\,z_j$ where $z_j:=j+d/2$, we also have
\[
\frac{\eta_{j+1}}{\lambda_{j+1}}=\frac{\eta_j+\frac1{z_j}}{\lambda_j+2\,z_j}<\frac{\eta_j}{\lambda_j}\,,
\]
where the inequality follows from
\[
z_j^2>\frac{\lambda_j}{2\,\eta_j}\quad\forall\,j\ge1\,.
\]
This inequality is indeed true for $j=1$ because $\eta_1=2/d$ and we obtain the result for any $j\ge1$ by induction using
\[
\eta_{j+1}-\eta_j=\frac2{d+2\,j}\ge\frac{\lambda_{j+1}}{2\,z_{j+1}^2}-\frac{\lambda_j}{2\,z_j^2}=2\,\frac{4\,j^2+2\,(d^2+2)\,j+d^3}{(d+2\,j)^2\,(d+2+2\,j)^2}\,.
\]
Altogether, for any $k\ge1$, we have
\[
d\,\mathcal E_2[F]\le\isd{|\nabla\Pi_k F|^2}+\frac{d\,\eta_{k+1}}{(k+1)\,(k+d)}\isd{|\nabla F|^2}
\]
and the constant in~\eqref{BVVimproved22} is given by
\[
\mathscr C_{d,2,k}=1-\frac{d\,\eta_{k+1}}{(k+1)\,(k+d)}\,.
\]
In the framework of our method, this estimate of the constant cannot be improved as shown by a test function such that $F_j=0$ for any $j\in\N$ such that $j\neq 0$ and $j\neq k+1$, but again this does not prove the optimality of $\mathscr C_{d,p,k}$.\end{proof}

%%%%%%%%%%%%%%%%%%%%%%%%%%%%%%%%%%%%%%%%%%%%%%%%%%%%%%%%%%%%%%%%%%%%%%%%%%
%%%%%%%%%%%%%%%%%%%%%%%%%%%%%%%%%%%%%%%%%%%%%%%%%%%%%%%%%%%%%%%%%%%%%%%%%%
\section{Improvements by the \emph{carr\'e du champ} method}\label{Sec:Improved-Carre}

We improve upon Frank's stability result in~\cite{Frank_2022} by giving a constructive estimate based on the \emph{carr\'e du champ} method, without assuming any additional constraint. Various computations that are needed for a complete proof, most of them already known in the literature, are collected in Appendix~\ref{Appendix:BE-Sphere}.

%%%%%%%%%%%%%%%%%%%%%%%%%%%%%%%%%%%%%%%%%%%%%%%%%%%%%%%%%%%%%%%%%%%%%%%%%%
\subsection{A simple estimate based on the heat flow, below the Bakry--Emery exponent}\label{Sec:BE}~

Let us consider the constant $\gamma$ given by
\be{gamma1}
\gamma:=\(\frac{d-1}{d+2}\)^2\,(p-1)\,(2^\#-p)\quad\mbox{if}\quad d\ge2\,,\quad\gamma:=\frac{p-1}3\quad\mbox{if}\quad d=1\,,
\ee
where~$2^\#:=\frac{2\,d^2+1}{(d-1)^2}$ is the \emph{Bakry--Emery exponent}. Notice that $\gamma=2-p$ with $1\le p\le2^\#$ means that
\[
\begin{array}{ll}
\displaystyle d=1\quad\mbox{and}\quad p=7/4=p_*(1)\,,\\
\displaystyle d>1\quad\mbox{and}\quad p=p_*(d):=\displaystyle\frac{3+d+2\,d^2-2\,\sqrt{4\,d+4\,d^2+d^3}}{(d-1)^2}\,.
\end{array}
\]
Let us define
\be{sstar}
s_\star:=\frac1{p-2}\quad\mbox{if}\quad p>2\quad\mbox{and}\quad s_\star:=+\infty\quad\mbox{if}\quad p\le2\,.
\ee
For any $s\in[0,s_\star)$, let
\be{phifunction}
\begin{array}{ll}
\varphi(s)=\frac{1-(p-2)\,s-\(1-(p-2)\,s\)^{-\frac\gamma{p-2}}}{2-p-\gamma}\quad&\mbox{if}\quad\gamma\neq2-p\quad\mbox{and}\quad p\neq2\,,\\[4pt]
\varphi(s)=\frac1{2-p}\(1+(2-p)\,s\)\log\(1+(2-p)\,s\)\quad&\mbox{if}\quad\gamma=2-p\neq0\,,\\[4pt]
\varphi(s)=\frac1\gamma\(e^{\gamma\,s}-1\)\quad&\mbox{if}\quad p=2\,.
\end{array}
\ee
In~\cite[Theorem~2.1]{Dolbeault_2020b} (also see~\cite{DEKL} and earlier related references therein) the \emph{improved Gagli\-ardo-Nirenberg inequalities}
\be{improved}
\nrms{\nabla F}2^2\ge d\,\varphi\(\frac{\mathcal E_p[F]}{\nrms Fp^2}\)\,\nrms Fp^2\quad\forall\,F\in\mathrm H^1(\S^d)
\ee
are stated with $\gamma$ given by~\eqref{gamma1} under the conditions
\[
d\ge1\quad\mbox{and}\quad1\le p\le2^\#\quad\mbox{if}\quad d\ge2\,,\quad p\ge1\quad\mbox{if}\quad d=1\,.
\]
Why this estimate is based on the heat flow is explained in Appendix~\ref{Appendix:BE-Sphere}. Additional justifications and discussion of the case $p=2$ are also given in Appendix~\ref{Appendix:BE-Sphere}.

Since $\varphi(0)=0$, $\varphi'(0)=1$, and $\varphi$ is convex increasing, with an asymptote at $s=s_\star$ if $p\in(2,2^\#)$, we know that $\varphi:[0,s_\star)\to\R^+$ is invertible and $\psi:\R^+\to[0,s_\star)$, $s\mapsto\psi(s):=s-\varphi^{-1}(s)$, is convex increasing with $\psi(0)=\psi'(0)=0$, $\lim_{t\to+\infty}\big(t-\psi(t)\big)=s_\star$, and
\[
\psi''(0)=\varphi''(0)=\frac{(d-1)^2}{(d+2)^2}\(2^\#-p\)(p-1)>0\quad\forall\,p\in(1,2^\#)\,.
\]
%-------------------------------------------------------------------------
\begin{proposition}\label{Prop:heat} With the above notation, $d\ge1$ and $p\in(1,2^\#)$, we have
\[
\nrms{\nabla F}2^2-d\,\mathcal E_p[F]\ge d\,\nrms Fp^2\,\psi\(\frac1d\,\frac{\nrms{\nabla F}2^2}{\nrms Fp^2}\)\quad\forall\,F\in\mathrm H^1(\S^d)\,.
\]
\end{proposition}
%-------------------------------------------------------------------------
If $p=2$, notice that $\psi$ is explicit and given by
\[
\psi(t):=t-\frac1\gamma\,\log(1+\gamma\,t)\quad\forall\,t\ge0\,.
\]
The proof of Proposition~\ref{Prop:LSI-BE} follows from the observation that $\psi(t)\ge\frac\gamma2\,\frac{t^2}{1+\gamma\,t}$ for any $t\ge0$.

%%%%%%%%%%%%%%%%%%%%%%%%%%%%%%%%%%%%%%%%%%%%%%%%%%%%%%%%%%%%%%%%%%%%%%%%%%
\bigskip\subsection{An estimate based on the fast diffusion flow, valid up to the critical exponent}~

The subcritical range $p\in[2^\#,2^*)$ corresponding to exponents between the \emph{Bakry--Emery exponent} and the critical Sobolev exponent is not covered in Section~\ref{Sec:BE}. In that case, we rely on entropy methods based on a fast diffusion or porous medium equation of exponent~$m$, which are detailed in Appendix~\ref{Appendix:BE-Sphere} (with corresponding references), to establish that an improved inequality~\eqref{improved} holds for any $\varphi=\varphi_{m,p}$, where
\be{phifunction:mp}
\varphi_{m,p}(s):=\int_0^{s}\exp\left[-\,\zeta\(\(1\,-\,(p-2)\,z\)^{1-\delta}-\(1\,-\,(p-2)\,s\)^{1-\delta}\)\right]\,dz\,,
\ee
provided $m\in\mathscr A_p:=\mathscr A_p:=\left\{m\in[m_-(d,p),m_+(d,p)]\,:\,\tfrac2p\le m<1\;\mbox{if}\; p<4\right\}$, where
\be{admissible.m}
m_\pm(d,p):=\frac1{(d+2)\,p}\(d\,p+2\pm\sqrt{d\,(p-1)\,\big(2\,d-(d-2)\,p\big)}\)\,,
\ee
while the parameters $\delta$ and $\zeta$ are defined by
\begin{align*}
&\delta:=1+\frac{(m-1)\,p^2}{4\,(p-2)}\,,\\
&\zeta:=\frac{
(d+2)^2\,p^2\,m^2-2\,p\,(d+2)\,(d\,p+2)\,m+d^2\(5\,p^2-12\,p+8\)+4\,d\,(3-2\,p)\,p+4}{(1-m)\,(d+2)^2\,p^2}\kern-1pt\,.
\end{align*}
Let $s_\star:=1/(p-2)$ as in~\eqref{sstar} and consider the inverse function $\varphi_{m,p}^{-1}:\R^+\to[0,s_\star)$ and $\psi_{m,p}(s):=s-\varphi_{m,p}^{-1}(s)$. Exactly as in the case $m=1$, we have the improved entropy -- entropy production inequality
\[
\nrms{\nabla F}2^2\ge d\,\nrms Fp^2\,\varphi_{m,p}\(\frac{\mathcal E_p[F]}{\nrms Fp^2}\)\quad\forall\,F\in\mathrm H^1(\S^d)\,,
\]
which provides us with the following stability estimate.
%-------------------------------------------------------------------------
\begin{proposition}\label{Prop:BE} With above notation, $d\ge1$, $p\in(2,2^*)$ and $m\in\mathscr A_p$, we have
\[
\nrms{\nabla F}2^2-d\,\mathcal E_p[F]\ge d\,\nrms Fp^2\,\psi_{m,p}\(\frac{\nrms{\nabla F}2^2}{d\,\nrms Fp^2}\)\quad\forall\,F\in\mathrm H^1(\S^d)\,.
\]
\end{proposition}
%-------------------------------------------------------------------------
The function $\varphi_{m,p}$ can be expressed in terms of the \emph{incomplete $\Gamma$ function}, while $\psi_{m,p}$ is known only implicitly.

%%%%%%%%%%%%%%%%%%%%%%%%%%%%%%%%%%%%%%%%%%%%%%%%%%%%%%%%%%%%%%%%%%%%%%%%%%
\medskip\subsection{Comparison with other estimates}~

Let us assume that $p\in(2,2^*)$. In~\cite{Frank_2022}, Frank proves the existence of a positive constant $\mathsf c_\star(d,p)$ such that
\[
\nrms{\nabla F}2^2-d\,\mathcal E_p[F]\ge\mathsf c_\star(d,p)\,\frac{\(\nrms{\nabla F}2^2+\nrms{F-\overline F}2^2\)^2}{\nrms{\nabla F}2^2+\frac d{p-2}\,\nrms F2^2}\quad\forall\,F\in\mathrm H^1(\S^d,d\mu)\,,
\]
where $\overline F:=\isd F$, which in particular implies the existence of a positive constant $\mathsf c(d,p)$ such that
\be{ImprovedStar}
\nrms{\nabla F}2^2-d\,\mathcal E_p[F]\ge\mathsf c(d,p)\,\frac{\nrms{\nabla F}2^4}{\nrms{\nabla F}2^2+\frac d{p-2}\,\nrms F2^2}\quad\forall\,F\in\mathrm H^1(\S^d,d\mu)\,,
\ee
for all $p\in (2,2^*)$. The value of the constant $\mathsf c_\star(d,p)$ found in~\cite{Frank_2022} is unknown as it follows from a compactness argument, in the spirit of~\cite{MR1124290}, but the exponent $4$ in the right-hand side of~\eqref{ImprovedStar}~is optimal. 
With
the test functions $F_\varepsilon(x)=1+\varepsilon\,x\cdot\nu$ for some arbitrary $\nu\in\S^d$, we can indeed check that
\[
\lim_{\varepsilon\to0}\frac1{\varepsilon^4}\(\nrms{\nabla F_\varepsilon}2^2-d\,\mathcal E_p[F_\varepsilon]\)=\frac{(d+p)\,(p-1)}{2\,d\,(d+3)}\,,
\]
which gives the upper bounds
\[
\mathsf c(d,p)\le\frac{(p-1)\,(d+p)}{2\,(p-2)\,(d+3)}\quad\mbox{and}\quad\mathsf c_\star(d,p)\le\frac{d^2}{(d+1)^2}\,\frac{(p-1)\,(d+p)}{2\,(p-2)\,(d+3)}\,.
\]
Let us notice that $\nrms{\nabla F}2^2\ge d\,\nrms{F-\overline F}2^2$ by the Poincar\'e inequality, so that we have
\[
\(\nrms{\nabla F}2^2+\nrms{F\kern-0.5pt-\kern-0.5pt\overline F}2^2\)^2\ge\nrms{\nabla F}2^4\ge\frac{d^2}{(d+1)^2}\(\nrms{\nabla F}2^2+\nrms{F\kern-0.5pt-\kern-0.5pt\overline F}2^2\)^2
\]
and, at least if $\mathsf c_\star(d,p)$ and $\mathsf c(d,p)$ are the optimal constants,
\[
\frac{d^2}{(d+1)^2}\,\mathsf c(d,p)\le\mathsf c_\star(d,p)\le\mathsf c(d,p)\,.
\]

 We claim that the \emph{carr\'e du champ} method provides us with a constructive estimate of $\mathsf c(d,p)$. Let
\[
\phi_c(s):=\frac d{2\,(1-c)}\(2\,c\,s-s_\star+\sqrt{s_\star^2+4\,c\,s\,(s-s_\star)}\)\,.
\]
%-------------------------------------------------------------------------
\begin{corollary}\label{Cor:BE} Let $p\in(2,2^*)$. With the notation of Proposition~\ref{Prop:BE}, inequality~\eqref{ImprovedStar} holds with
\[
\mathsf c=\sup\Big\{c>0\,:\,\exists\,m\in\mathscr A_p\;\mbox{such that}\;\phi_c(s)\le\varphi_{m,p}(s)\;\forall\,s\in[0,s_\star)\Big\}\,.
\]
\end{corollary}
%-------------------------------------------------------------------------
\begin{proof} With no loss of generality, let us assume that $\nrms Fp=1$ and define
\[
\mathsf i=\nrms{\nabla F}2^2\quad\mbox{and}\quad\mathsf e:=\frac{1-\nrms F2^2}{p-2}\,,
\]
so that $\nrms F2^2=1-(p-2)\,\mathsf e$. With $c=\mathsf c(d,p)$, we can rewrite~\eqref{ImprovedStar} as
\[
\mathsf i-d\,\mathsf e\ge\frac{c\,\mathsf i^2}{\mathsf i+\frac d{p-2}-d\,\mathsf e}\,,
\]
which amounts to
\[
\mathsf i-d\,\mathsf e\ge\phi_c(\mathsf e)\,.
\]
Since we know that $\mathsf i-d\,\mathsf e\ge\varphi_{m,p}(\mathsf e)$, the conclusion follows for the largest possible $c>0$ such that $\varphi_{m,p}\ge\phi_c$. \end{proof}

%%%%%%%%%%%%%%%%%%%%%%%%%%%%%%%%%%%%%%%%%%%%%%%%%%%%%%%%%%%%%%%%%%%%%%%%%%
%%%%%%%%%%%%%%%%%%%%%%%%%%%%%%%%%%%%%%%%%%%%%%%%%%%%%%%%%%%%%%%%%%%%%%%%%%
\section{Global stability results}\label{Sec:Global}

We collect the statements of Theorems~\ref{Thm:sphere-LSI-stability} and~\ref{Thm:sphere-GNS-stability} into a single result. The whole section is devoted to its proof.
%-------------------------------------------------------------------------
\begin{theorem}\label{Thm:sphere-GNS-LSI-stability} Let $d\ge1$ and $p\in(1,2^*)$. For any $F\in\mathrm H^1(\S^d,d\mu)$, we have
\be{Ineq:GNS-LSI-stability}
\isd{|\nabla F|^2}-d\,\mathcal E_p[F]\ge\mathscr S_{d,p}\(\frac{\nrms{\nabla\Pi_1 F}2^4}{\nrms{\nabla F}2^2+\nrms F2^2}+\nrms{\nabla(\mathrm{Id}-\Pi_1)\,F}2^2\)
\ee
for some explicit stability constant $\mathscr S_{d,p}>0$.
\end{theorem}
%-------------------------------------------------------------------------
\noindent The value of $\mathscr S_{d,p}$ is elementary and explicit but its expression is lengthy. We explain in the proof how to compute it with all necessary details to obtain a numerical expression for $\mathscr S_{d,p}$ for given $p$ and $d$, if needed. 

\begin{proof}[Proof of Theorem~\ref{Thm:sphere-GNS-LSI-stability}] By homogeneity of~\eqref{Ineq:GNS-LSI-stability}, we can assume that $\nrms F2=1$ without loss of generality. For clarity, we subdivide the proof into various steps. Let us start with the case $p>2$.

%%%%%%%%%%%%%%%%%%%%%%%%%%%%%%%%%%%%%%%%%%%%%%%%%%%%%%%%%%%%%%%%%%%%%%%%%%
\step{An estimate based on the \emph{carr\'e du champ} method}
If $\nrms{\nabla F}2^2/\nrms Fp^2\ge\vartheta_0>0$, we know by the convexity of $\psi_{m,p}$ that
\be{S1}
\nrms{\nabla F}2^2-d\,\mathcal E_p[F]\ge d\,\nrms Fp^2\,\psi_{m,p}\(\frac1d\,\tfrac{\nrms{\nabla F}2^2}{\nrms Fp^2}\)\ge\frac d{\vartheta_0}\,\psi_{m,p}\(\frac{\vartheta_0}d\)\,\nrms{\nabla F}2^2\,.
\ee
In that case, we conclude from $\nrms{\nabla F}2^2=\nrms{\nabla\Pi_1 F}2^2+\nrms{\nabla(\mathrm{Id}-\Pi_1)\,F}2^2$ and
\[
\nrms{\nabla\Pi_1 F}2^2\ge\frac{\nrms{\nabla\Pi_1 F}2^4}{\nrms{\nabla F}2^2+\nrms Fp^2}\,.
\]
Let us assume now that $\nrms{\nabla F}2^2<\vartheta_0\,\nrms Fp^2$. By taking into account~\eqref{GNS}, we obtain
\[
\nrms{\nabla F}2^2<\vartheta_0\,\nrms Fp^2\le\vartheta_0\(\nrms F2^2+\frac{p-2}d\,\nrms{\nabla F}2^2\)\,.
\]
Using $\nrms F2=1$, under the assumption that $\vartheta_0<d/(p-2)$, we know that
\be{hyp1}
\vartheta:=\nrms{\nabla F}2^2<\frac{d\,\vartheta_0}{d-(p-2)\,\vartheta_0}\,.
\ee
Notice that the parameter $\vartheta_0$ still has to be chosen.

%%%%%%%%%%%%%%%%%%%%%%%%%%%%%%%%%%%%%%%%%%%%%%%%%%%%%%%%%%%%%%%%%%%%%%%%%%
\step{An estimate of the average}
Let us estimate $\Pi_0 F:=\isd F$. By the Poincar\'e inequality, we have
\[
1=\nrms F2^2=\(\isd F\)^2+\nrms{(\mathrm {Id}-\Pi_0)\,F}2^2\le\(\isd F\)^2+\frac\vartheta d\,,
\]
and on the other hand we know that $\(\isd F\)^2\le\nrms F2^2=1$ by the Cauchy--Schwarz inequality, so that
\be{average}
\frac{d-\vartheta}d<\(\isd F\)^2\le1\,.
\ee
We assume in the sequel that
\be{hyp2}
\vartheta<d\,.
\ee 

%%%%%%%%%%%%%%%%%%%%%%%%%%%%%%%%%%%%%%%%%%%%%%%%%%%%%%%%%%%%%%%%%%%%%%%%%%
\step{Partial decomposition on spherical harmonics}
With no loss of generality, let us write
\be{decomposition}
F=\mathscr M\(1+\varepsilon\,\mathscr Y+\eta\,G\)
\ee
such that $\mathscr M=\Pi_0F$ and $\Pi_1F=\varepsilon\,\mathscr M \,\mathscr Y$ where $\mathscr Y(x)=\sqrt{\frac{d+1}d}\,x\cdot\nu$ for some given $\nu\in\S^d$. Here the functions $\mathscr Y$ and $G$ are normalized so that $\nrms{\nabla\mathscr Y}2=\nrms{\nabla G}2=1$ and
\[
\mathscr M^{-2}\,\nrms{\nabla F}2^2=\varepsilon^2+\eta^2=\vartheta\quad\mbox{and}\quad\mathscr M^{-2}\,\nrms F2^2=1+\frac1d\,\varepsilon^2+\eta^2\,\nrms G2^2\,.
\]
We observe that $\Pi_0(F-\mathscr M)=0$. Using~\eqref{GNS} and the Poincar\'e inequality, we have
\[
\nrms{F-\mathscr M}p^2\le\nrms{F-\mathscr M}2^2+\frac{p-2}d\,\nrms{\nabla F}2^2\le\frac{p-1}d\,\nrms{\nabla F}2^2\,.
\]
Similarly, by~\eqref{BVVimproved22}, \emph{i.e.},
\[
\frac{d\,p}{2\,(d+p)}\,\nrms{\nabla G}2^2=\(1-\frac{2\,d-p\,(d-2)}{2\,(d+p)}\)\nrms{\nabla G}2^2\ge d\,\mathcal E_p[G]\,,
\]
and the improved Poincar\'e inequality~\eqref{BVVimproved22} written with $p=1$ and $k=1$
\[
\nrms G2^2\le\frac1{2\,(d+1)}\,\nrms{\nabla G}2^2=\frac1{2\,(d+1)}\,,
\]
we have
\[
\nrms Gp^2\le\nrms G2^2+\frac{p\,(p-2)}{2\,(d+p)}\,\nrms{\nabla G}2^2\le C_{p,d}
\]
using $\nrms{\nabla G}2=1$, with $C_{p,d}:=\frac1{2\,(d+1)}+\frac{p\,(p-2)}{2\,(d+p)}$. By the Cauchy--Schwarz inequality, we also have
\[
\nrm G1\le\frac1{\sqrt{2\,(d+1)}}\,,
\]

We recall that the eigenvalues of $-\Delta$ on $\S^d$ are $\lambda_k=k\,(k+d-1)$ with $k\in\N$. In preparation for a detailed Taylor expansion as in~\cite{Frank_2022}, let us consider the function
\[
\mathscr Y(x):=\sqrt{\frac{d+1}d}\,x\cdot\nu\,,
\]
which is such that $-\Delta\mathscr Y=\lambda_1\,\mathscr Y$ with $\lambda_1=d$ and
\begin{align*}
&\nrms{\nabla\mathscr Y}2^2=1\,,\quad\nrms{\mathscr Y}2^2=\frac1d\,,\\
&\nrms{\mathscr Y}4^4=\frac{3\,(d+1)}{(d+3)\,d^2}\,,\quad\nrms{\mathscr Y}6^6=\frac{15\,(d+1)^2}{(d+3)\,(d+5)\,d^2}\,.
\end{align*}
The function $\mathscr Y_2:=\mathscr Y^2-\frac1d$ is such that $-\Delta\mathscr Y_2=\lambda_2\,\mathscr Y_2$ with $\lambda_2=2\,(d+1)$ and 
\[
\nrms{\mathscr Y_2}2^2=\frac2{d\,(d+3)}\,,\quad\nrms{\nabla \mathscr Y_2}2^2=\frac{4\,(d+1)}{d\,(d+3)}\,.
\]
The function $\mathscr Y_3:=\mathscr Y^3-\frac{3\,(d+1)}{d\,(d+3)}\,\mathscr Y$ is such that $-\Delta\mathscr Y_3=\lambda_3\,\mathscr Y_3$ with $\lambda_3=3\,(d+2)$ and 
\[
\nrms{\mathscr Y_3}2^2=\frac{6\,(d+1)^2}{(d+5)\,(d+3)^2\,d^2}\,,\quad\nrms{\nabla \mathscr Y_3}2^2=\frac{18\,(d+2)\,(d+1)^2}{(d+5)\,(d+3)^2\,d^2}\,.
\]

As a consequence of~\eqref{decomposition}, we know that $\Pi_0G=\Pi_1G=0$ and $\nrms{\nabla G}2=1$. 
Let
\[
g_2:=\frac{\isd{\nabla\mathscr Y_2\cdot\nabla G}}{\nrms{\nabla\mathscr Y_2}2}\quad\mbox{and}\quad g_3:=\frac{\isd{\nabla\mathscr Y_3\cdot\nabla G}}{\nrms{\nabla\mathscr Y_3}2}\,.
\]
With $k=1$, $2$, using $-\Delta\mathscr Y_k=\lambda_k\,\mathscr Y_k$ with $\lambda_k=\nrms{\nabla\mathscr Y_k}2^2/\nrms{\mathscr Y_k}2^2$, we compute
\[
\isd{\mathscr Y^k\,G}=\isd{\mathscr Y_k\,G}=\frac{\nrms{\mathscr Y_k}2^2}{\nrms{\nabla\mathscr Y_k}2^2}\isd{\nabla\mathscr Y_k\cdot\nabla G}=g_k\,\frac{\nrms{\mathscr Y_k}2^2}{\nrms{\nabla\mathscr Y_k}2}
\]
and obtain
\[
\isd{\mathscr Y^2\,G}=\isd{\mathscr Y_2\,G}=\frac{g_2}{\sqrt{d\,(d+1)\,(d+3)}}\,,
\]
\[
\isd{\mathscr Y^3\,G}=\isd{\mathscr Y_3\,G}=c_3\,g_3\quad\mbox{with}\quad c_3:=\frac{d+1}{d\,(d+3)}\,\sqrt{\frac2{(d+2)\,(d+5)}}\,.
\]

%%%%%%%%%%%%%%%%%%%%%%%%%%%%%%%%%%%%%%%%%%%%%%%%%%%%%%%%%%%%%%%%%%%%%%%%%%
\step{Taylor expansions (1)}
Let us start with elementary estimates of $\nrms{1+\varepsilon\,\mathscr Y}p$. If it holds that $2\le p<3$ and $|s|<1$, we have
\[
\frac12\,\Big((1+s)^p+(1-s)^p\Big)\le1+\frac p2\,(p-1)\,s^2\(1+\frac1{12}\,(p-2)\,(p-3)\,s^2\)
\]
because all other terms in the series expansion of the left-hand side~around $s=0$ correspond to even powers of $s$ and appear with nonpositive coefficients. If either $1\le p<2$ or $p>3$ and $|s|<1/2$, let
\[
f_p(s):=\frac12\,\Big((1+s)^p+(1-s)^p\Big)-\(1+\frac p2\,(p-1)\,s^2\)
\]
and notice that $f_p''(s)=\frac p2\,(p-1)\,\big((1+s)^{p-2}+(1-s)^{p-2}-2\big)\ge0$ by convexity of the function $y\mapsto y^{p-2}$ so that $c_p^{(+)}$ defined as the maximum of $s\mapsto f_p(s)/s^6$ on $[-1/2,1/2]\ni s$ is finite and we have
\be{cp+}
\frac12\,\Big((1+s)^p+(1-s)^p\Big)\le1+\frac p2\,(p-1)\,s^2\(1+\frac1{12}\,(p-2)\,(p-3)\,s^2\)+c_p^{(+)}\,s^6\,.
\ee
We adapt the convention that $c_p^{(+)}=0$ if $p\in[2,3)$. Using the fact that $\mathscr Y(-x)=-\,\mathscr Y(x)$,
\[
\nrms{1+\varepsilon\,\mathscr Y}p^p=\frac12\(\nrms{1+\varepsilon\,\mathscr Y}p^p+\nrms{1-\varepsilon\,\mathscr Y}p^p\)\,.
\]
For any $\varepsilon\in(0,1/2)$ we use~\eqref{cp+} to write 
\begin{multline*}
\nrms{1+\varepsilon\,\mathscr Y}p^p\!-\(1+\frac p2\,(p-1)\(\nrms{\mathscr Y}2^2+\frac1{12}\,(p-2)\,(p-3)\,\nrms{\mathscr Y}4^4\varepsilon^2\)\varepsilon^2\)\\
\le c_p^{(+)}\,\nrms{\mathscr Y}6^6\varepsilon^6\,.
\end{multline*}
For similar reasons, one can prove that there is another constant $c_p^{(-)}$ which provides us with a lower bound $c_p^{(-)}\,\nrms{\mathscr Y}6^6\,\varepsilon^6$. Altogether, this amounts to
\be{s3}
c_{p,d}^{(-)}\,\varepsilon^6\le\nrms{1+\varepsilon\,\mathscr Y}p^p-\(1+a_{p,d}\,\varepsilon^2+b_{p,d}\,\varepsilon^4\)\le c_{p,d}^{(+)}\,\varepsilon^6\,,
\ee
with
\begin{multline*}
a_{p,d}:=\frac{p\,(p-1)}{2\,d}\,,\quad b_{p,d}:=\frac14\,(p-2)\,(p-3)\,\frac{d+1}{d\,(d+3)}\,a_{p,d}\,,\\
c_{p,d}^{(\pm)}:=\frac{15\,(d+1)^2}{(d+3)\,(d+5)\,d^2}\,c_p^{(\pm)}\,.
\end{multline*}
Estimate~\eqref{s3} is valid under the condition that $\varepsilon<1/2$. We shall therefore request that
\be{hyp3}
\vartheta<\frac14\,,
\ee
which is an obvious sufficient condition according to~\eqref{hyp1}. Now we draw two consequences of~\eqref{s3}. First, let us give an upper estimate of $\nrms{1+\varepsilon\,\mathscr Y}p^2$. Using
\[
(1+s)^\frac2p\le1+2\,\frac sp-(p-2)\,\frac{s^2}{p^2}+\frac23\,(p-1)\,(p-2)\,\frac{s^3}{p^3}\,,
\]
we obtain
\be{s3-1}
\nrms{1+\varepsilon\,\mathscr Y}p^2\le1+\frac2p\,a_{p,d}\,\varepsilon^2+\frac1{p^2}\(2\,p\,b_{p,d}-(p-2)\,a_{p,d}^2\)\,\varepsilon^4+r^{(+)}\,\varepsilon^6\,,
\ee
where the remainder term $r^{(+)}$ is explicitly estimated by 
\begin{multline*}
96\,p^3\,r^{(+)}=64\,a_{p,d}^3\left(p^2-3\,p+2\right)+48\,a_{p,d}^2\left(p^2-3\,p+2\right)(2\,b_{p,d}+c_{p,d})\\
\kern12pt+12\,a_{p,d}\,(p-2)\,(2\,b_{p,d}+c_{p,d})\,(2\,b_{p,d}\,(p-1)+c_{p,d}\,(p-1)-8\,p)\\
+8\,b_{p,d}^3\left(p^2-3\,p+2\right)+12\,b_{p,d}^2\,(p-2)\,(c_{p,d}\,(p-1)-4\,p)\\
+6\,b_{p,d}\,c_{p,d}\,(p-2)\,(c_{p,d}\,(p-1)-8\,p)\\
+c_{p,d}\left(c_{p,d}^2\left(p^2-3\,p+2\right)-12\,c_{p,d}\,(p-2)\,p+192\,p^2\right)\,.
\end{multline*}
To do this estimate, we simply write that $\varepsilon^\alpha\le2^{6-\alpha}\,\varepsilon^6$ for any $\alpha>6$ using the (nonoptimal) bound $\varepsilon^2<1/2$. Similarly, using
\begin{multline*}
(1+s)^{\frac2p-1}\le1-(p-2)\,\frac sp+(p-1)\,(p-2)\,\frac{s^2}{p^2}-\frac13\,(p-1)\,(p-2)\,(3\,p-2)\frac{s^3}{p^3}\\
+\frac16\,(p-1)\,(p-2)\,(3\,p-2)\,(2\,p-1)\frac{s^4}{p^4}\,,
\end{multline*}
we obtain
\be{s3-2}
\nrms{1+\varepsilon\,\mathscr Y}p^{2-p}\le1+\frac{p-2}p\,a_{p,d}\,\varepsilon^2-\frac{p-2}{p^2}\(p\,b_{p,d}-(p-1)\,a_{p,d}^2\)\,\varepsilon^4+r^{(-)}\,\varepsilon^6\,,
\ee
where the remainder term $r^{(-)}$ also has an explicit expression in terms of $a_{p,d}$, $b_{p,d}$ and $c_{p,d}^{(-)}$, which is not given here.

%%%%%%%%%%%%%%%%%%%%%%%%%%%%%%%%%%%%%%%%%%%%%%%%%%%%%%%%%%%%%%%%%%%%%%%%%%
\step{Taylor expansions (2)}
With $u\ge0$, $u+r\ge0$ and $p>2$, we claim that
\[\
(u+r)^p\le u^p+p\,u^{p-1}\,r+\frac p2\,(p-1)\,u^{p-2}\,r^2+\sum_{2<k<p}C_k^p\,u^{p-k}\,|r|^k+K_p\,|r|^p
\]
for some constant $K_p>0$, where the coefficients
\[
C_k^p:=\frac{\Gamma(p+1)}{\Gamma(k+1)\,\Gamma(p-k+1)}
\]
are the binomial coefficients if $p$ is an integer. It is proved in~\cite{https://doi.org/10.48550/arxiv.2209.08651} that $K_p=1$ if $p\in(2,4]\cup\{6\}$. The proof is similar to the above analysis and is left to the reader. Let us integrate this inequality and raise both sides to the power $2/p$ to get
\[
\nrms{u+r}p^2\le\nrms up^2\,(1+s)^\frac2p\,,
\]
with
\begin{multline*}
s=\frac1{\nrms up^p}\(p\isd{u^{p-1}\,r}+\frac p2\,(p-1)\isd{u^{p-2}\,r^2}\right.\\
\left.+\sum_{2<k<p}C_k^p\isd{u^{p-k}\,|r|^k}+K_p\isd{|r|^p}\)\,.
\end{multline*}
By assumption $2/p<1$ so that we may use the identity $(1+s)^{2/p}\le1+2\,s/p$ for any $s\ge-1$. Notice that we can assume that $u+r\ge0$ and deduce from (1) that $s\ge-1$. As a consequence, we have
\begin{multline*}
\nrms{u+r}p^2\le\nrms up^2\\
\hspace*{2cm}+\frac2p\,\nrms up^{2-p}\,\Bigg(p\isd{u^{p-1}\,r}+\frac p2\,(p-1)\isd{u^{p-2}\,r^2}\\
+\sum_{2<k<p}C_k^p\isd{u^{p-k}\,|r|^k}+K_p\isd{|r|^p}\Bigg)\,.
\end{multline*}
We apply these computations to $u=1+\varepsilon\,\mathscr Y$ and $r=\eta\,G$ to obtain
\begin{align*}
&\mathscr M^{-2}\,\nrms Fp^2-\nrms{1+\varepsilon\,\mathscr Y}p^2\\
&\le\frac2p\,\nrms{1+\varepsilon\,\mathscr Y}p^{2-p}\,\eta\,\Bigg(p\isd{\(1+\varepsilon\,\mathscr Y\)^{p-1}\,G}\\
&\hspace*{3.4cm}+\frac p2\,(p-1)\,\eta\isd{\(1+\varepsilon\,\mathscr Y\)^{p-2}\,|G|^2}\\
&\hspace*{3.4cm}+\sum_{2<k<p}C_k^p\,\eta^k\isd{\(1+\varepsilon\,\mathscr Y\)^{p-k}\,|G|^k}+K_p\,\eta^p\isd{|G|^p}\Bigg)\,.
\end{align*}
Let us detail the expansion of each of the terms involving $G$ in the right-hand side of this estimates. For any $s\in(-1/2,1/2)$, using the expansion
\[
(1+s)^{p-1}\le1+(p-1)\,s+\frac12\,(p-1)\,(p-2)\,s^2+\frac16\,(p-1)\,(p-2)\,(p-3)\,s^3+R_p\,s^4
\]
for some constant $R_p>0$ applied with $s=1+\varepsilon\,\mathscr Y$, we obtain
\begin{multline*}
\eta\isd{\(1+\varepsilon\,\mathscr Y\)^{p-1}\,G}\le\frac12\,(p-1)\,(p-2)\,\frac{g_2}{\sqrt{d\,(d+1)\,(d+3)}}\,\eta\,\varepsilon^2\\
+\frac16\,(p-1)\,(p-2)\,(p-3)\,c_3\,g_3\,\eta\,\varepsilon^3+\frac{R_p\,\eta\,\varepsilon^4}{\sqrt{2\,(d+1)}}\,.
\end{multline*}
The other terms admit simpler expansions:
\begin{multline*}
\eta^2\isd{\(1+\varepsilon\,\mathscr Y\)^{p-2}\,|G|^2}\le\eta^2\,(1+\varepsilon)^{p-2}\,\nrms G2^2\\
\le\nrms G2^2\,\eta^2+\frac1{2\,(d+1)}\,\eta^2\((1+\varepsilon)^{p-2}-1\)
\end{multline*}
and
\begin{multline*}
\sum_{2<k<p}C_k^p\,\eta^k\isd{\(1+\varepsilon\,\mathscr Y\)^{p-k}\,|G|^k}+K_p\,\eta^p\isd{|G|^p}\\
\le\sum_{2<k<p}C_k^p\,\eta^k\(1+\varepsilon\)^{p-k}\,\nrms Gp^k+K_p\,\eta^p\,\nrms Gp^p\\
\le\sum_{2<k<p}C_k^p\,\eta^k\(1+\varepsilon\)^{p-k}\,C_{p,d}^{k/p}+K_p\,\eta^p\,C_{p,d}\,.
\end{multline*}
Collecting~\eqref{s3-1} and~\eqref{s3-2} with the above estimates, we arrive at
\begin{align*}
&\mathscr M^{-2}\,\nrms Fp^2\\
&\le1+\frac2p\,a_{p,d}\,\varepsilon^2+\frac1{p^2}\(2\,p\,b_{p,d}-(p-2)\,a_{p,d}^2\)\,\varepsilon^4+r^{(+)}\,\varepsilon^6\\
&\hspace*{12pt}+\(1+\frac{p-2}p\,a_{p,d}\,\varepsilon^2-\frac{p-2}{p^2}\(p\,b_{p,d}-(p-1)\,a_{p,d}^2\)\,\varepsilon^4+r^{(-)}\,\varepsilon^6\)\\
&\hspace*{24pt}\cdot\Bigg[(p-1)\,(p-2)\,\frac{g_2}{\sqrt{d\,(d+1)\,(d+3)}}\,\eta\,\varepsilon^2+\frac13\,(p-1)\,(p-2)\,(p-3)\,c_3\,g_3\,\eta\,\varepsilon^3\\
&\hspace*{1.8cm}+\frac{2\,R_p\,\eta\,\varepsilon^4}{\sqrt{2\,(d+1)}}+(p-1)\(\nrms G2^2\,\eta^2+\frac1{2\,(d+1)}\,\eta^2\((1+\varepsilon)^{p-2}-1\)\)\\
&\hspace*{4.8cm}+\frac2p\(\sum_{2<k<p}C_k^p\,\eta^k\(1+\varepsilon\)^{p-k}\,C_{p,d}^{k/p}+K_p\,\eta^p\,C_{p,d}\)\Bigg]\,.
\end{align*}
Using $|g_2|<1$, $|g_3|<1$, and $2\,(d+1)\,\nrms G2^2<1$, this gives rise to an explicit although lengthy expression for a positive constant $\mathcal R_{p,d}$ such that
\[
\mathscr M^{-2}\(\isd{|\nabla F|^2}-d\,\mathcal E_p[F]\)\ge A\,\varepsilon^4-B\,\varepsilon^2\,\eta+C\,\eta^2-\mathcal R_{p,d}\(\vartheta^p+\vartheta^{5/2}\)\,,
\]
with $A:=\frac{(p-1)\,(d+p)}{2\,d\,(d+3)}$, $B:=\frac{d\,(p-1)}{\sqrt{d\,(d+1)\,(d+3)}}$ and $C:=\frac{d+2}{2\,(d+1)}$. The discriminant
\[
B^2-4\,A\,C=-\,\frac1{d\,(d+3)}\,(p-1)\,\big(2\,d-p\,(d-2)\big)
\]
is negative if (and only if) $p\in(1,2^*)$, so that we can write
\[
A\,s^2-B\,s+C=(A-\lambda)\,s^2-B\,s+(C-\lambda)+\lambda\(s^2+1\)\ge\lambda\(s^2+1\)\,,
\]
where
\[
\lambda:=\frac12\(A+C+\sqrt{(A-C)^2+B^2}\)
\]
is given by the condition that $B^2-4\,(A-\lambda)\,(C-\lambda)=0$. Altogether, we obtain
\[
\mathscr M^{-2}\(\isd{|\nabla F|^2}-d\,\mathcal E_p[F]\)\ge\lambda\(\varepsilon^4+\eta^2\)-\mathcal R_{p,d}\(\vartheta^p+\vartheta^{5/2}\)\,.
\]

%%%%%%%%%%%%%%%%%%%%%%%%%%%%%%%%%%%%%%%%%%%%%%%%%%%%%%%%%%%%%%%%%%%%%%%%%%
\step{Conclusion if $p>2$}
We choose $\vartheta>0$ such that~\eqref{hyp2} and~\eqref{hyp3} are fulfilled. With the additional assumption that
\[
\vartheta\le\vartheta_{p,d}:=\Big\{\theta>0\,:\,\mathcal R_{p,d}\(\theta^p+\theta^{5/2}\)=\frac\lambda4\,\theta^2\Big\}\,,
\]
using $\eta^4\le\eta^2$ and $2\,\varepsilon^2\,\eta^2\le\varepsilon^4+\eta^2$ if $\eta<1$, we have
\[
\mathcal R_{p,d}\(\vartheta^p+\vartheta^{5/2}\)\le\frac\lambda4\,\vartheta^2=\frac\lambda4\(\varepsilon^2+\eta^2\)^2\le\frac\lambda2\(\varepsilon^4+\eta^2\)\le\frac\lambda2\(\frac{\varepsilon^4}{\varepsilon^2+\eta^2+1}+\eta^2\)\,.
\]
For any $F$ such that $\nrms{\nabla F}2^2=\vartheta$, we obtain
\[
\isd{|\nabla F|^2}-d\,\mathcal E_p[F]\ge\frac\lambda2\(\frac{\nrms{\nabla\Pi_1 F}2^4}{\nrms{\nabla F}2^2+\nrms Fp^2}+\nrms{\nabla(\mathrm{Id}-\Pi_1)\,F}2^2\)\,.
\]
Using~\eqref{S1} and~\eqref{average}, this completes the proof of Theorem~\ref{Thm:sphere-GNS-LSI-stability} if $p>2$ with
\[
\vartheta\le\min\left\{\frac d2,\,\frac14,\,\vartheta_{p,d}\right\}=\frac{d\,\vartheta_0}{d-(p-2)\,\vartheta_0}\quad\mbox{and}\quad\mathscr S_{d,p}=\min\left\{\frac d{\vartheta_0}\,\psi_{m,p}\(\frac{\vartheta_0}d\),\,\frac\lambda2\right\}\,.
\]

%%%%%%%%%%%%%%%%%%%%%%%%%%%%%%%%%%%%%%%%%%%%%%%%%%%%%%%%%%%%%%%%%%%%%%%%%%
\step{The case $p\le2$}
The strategy is the same, with some simplifications, so we only sketch the proof and emphasize the changes compared to the case $p>2$. Let us notice that
\[
(1+s)^p\le1+p\,s+\frac p2\,(p-1)\,s^2\quad\mbox{if}\quad1\le p<2
\]
and $(1+s)^2\,\log\big((1+s)^2\big)\le 2\,s+2\,s^2+\frac23\,s^3$ in the limit case $p=2$. The estimates involving $1+\varepsilon\,\mathscr Y$ are therefore essentially the same if we assume $\varepsilon<1/2$, while the computation of $\nrms{u+r}p^2$ is in fact simpler, when applied to $u=1+\varepsilon\,\mathscr Y$ and $r=\eta\,G$. The estimate on the average is simplified because $\nrms Fp\le\nrms F2$ by H\"older's inequality, since $d\mu$ is a probability measure on $\S^d$. Spectral estimates are exactly the same and the Taylor expansions present no additional difficulty, as we can use~\eqref{GNS} for some exponent $q\in(2,2^*)$ to control the remainder terms if $p=2$, so that $(1+s)^2\,\log\big((1+s)^2\big)\le 2\,s+2\,s^2+\kappa_q\,s^q$ for some $\kappa_q>0$. The conclusion is the same as for $p>2$ except that we have to replace $\psi_{m,p}$ by $\psi$ defined as in Proposition~\ref{Prop:heat}.
\end{proof}

%%%%%%%%%%%%%%%%%%%%%%%%%%%%%%%%%%%%%%%%%%%%%%%%%%%%%%%%%%%%%%%%%%%%%%%%%%
%%%%%%%%%%%%%%%%%%%%%%%%%%%%%%%%%%%%%%%%%%%%%%%%%%%%%%%%%%%%%%%%%%%%%%%%%%
\appendix\section{Improved Gaussian inequalities, hypercontractivity and stability}\label{AppendixB}

Whether the results of Theorems~\ref{Thm:sphere-LSI-improved},~\ref{StabSubcriticalSphereGNS} and~\ref{StabSubcriticalSphere} can be extended to Euclidean case with the Gaussian measure is a very natural question. Spherical harmonics can indeed be replaced by Hermite polynomials and there is a clear correspondence for spectral estimates. The answer is yes for a whole family of interpolation inequalities, but it is no for the logarithmic Sobolev inequality, which is an endpoint of the family.

Let us consider the \emph{normalized Gaussian measure} on $\R^d$ defined by
\[
d\sigma(x)=(2\,\pi)^{-\frac d2}\,e^{-\frac12\,|x|^2}\,dx\,.
\]
For any $p\in[1,2)$, Beckner in~\cite{MR954373} established the family of interpolation inequalities
\be{GaussianInterpolation}
\frac{\nrmG f2^2-\nrmG fp^2}{2-p}\le\nrmG{\nabla f}2^2\quad\forall\,f\in\mathrm H^1(\R^d,d\sigma)\,.
\ee
With $p=1$, inequality~\eqref{GaussianInterpolation} is the \emph{Gaussian Poincar\'e inequality} while one recovers the Gaussian logarithmic Sobolev inequality of~\cite{MR420249} in the limit as $p\to2$. For any $p\in[1,2)$, the inequality is optimal: using $f_\varepsilon:=1+\varepsilon\,\varphi$ as a test function, where $\varphi$ is such that \hbox{$\irdg\varphi=0$}, we recover the \emph{Gaussian Poincar\'e inequality} with optimal constant in the limit as $\varepsilon\to0$, so that the constant in~\eqref{GaussianInterpolation} cannot be improved. Based on~\cite{MR1796718,MR2375056}, the improved version of the inequality
\be{GaussianInterpolation-improved}
\frac{\nrmG f2^2-\nrmG fp^2}{2-p}\le\frac p2\,\nrmG{\nabla f}2^2\quad\forall\,f\in\mathrm H^1(\R^d,d\sigma)
\ee
holds under the additional condition
\be{Cdt}
\irdg{x\,f(x)}=0\,.
\ee
Let us give a short proof of~\eqref{GaussianInterpolation-improved}. Assume that $f=\sum_{k\in\N}f_k$ is a decomposition on Hermite functions such that $\mathcal Lf_k=-\,k\,f_k$ where $\mathcal L=\Delta-x\cdot\nabla$ is the Ornstein--Uhlenbeck operator, and let $a_k:=\nrmG{f_k}2^2$ for any $k\in\N$, so that
\[
\nrmG f2^2=\sum_{k\in\N}a_k\quad\mbox{and}\quad\nrmG{\nabla f}2^2=\sum_{k\in\N}k\,a_k\,.
\]
Let us consider the solution of
\be{OUeqn}
\frac{\partial u}{\partial t}=\mathcal L u
\ee
with initial datum $u(t=0,\cdot)=f$ and notice that
\[
\nrmG{u(t,\cdot)}2^2=\sum_{k\in\N}a_k\,e^{-2\,k\,t}\,.
\]
Hence, if~\eqref{Cdt} holds, $a_1=0$ and
\be{Gross}
\begin{aligned}
\nrmG f2^2-\nrmG{u(t,\cdot)}2^2&=\sum_{k\ge 2}a_k\(1-e^{-2\,k\,t}\)\\
&\le\frac12\(1-e^{-4\,t}\)\sum_{k\in\N}k\,a_k=\frac12\(1-e^{-4\,t}\)\nrmG{\nabla f}2^2
\end{aligned}
\ee
because $k\mapsto\(1-e^{-2\,k\,t}\)/k$ is monotone nonincreasing for any given $t\ge0$. Next, we use Nelson's hypercontractivity estimate in~\cite[Theorem~3]{MR0343816} to find $t_*>0$ such that
\[
\nrmG{u(t_*,\cdot)}2^2\le\nrmG fp^2\,.
\]
As noted in~\cite{MR420249}, this estimate can be seen as a consequence of the \emph{Gaussian logarithmic Sobolev inequality}
\be{GLSI}
\irdg{|v|^2\,\log\(\frac{|v|^2}{\nrmG v2^2}\)}\le2\irdg{|\nabla v|^2}\quad\forall\,v\in\mathrm H^1(\R^d,d\sigma)\,,
\ee
and the argument goes as follows. With $h(t):=\nrmG{u(t,\cdot)}{q(t)}$ for some exponent~$q$ depending on $t$ and $u$ solving~\eqref{OUeqn}, we have
\[
\frac{h'}h=\frac{q'}{q^2}\irdg{\frac{|u|^q}{h^q}\,\log\(\frac{|u|^q}{h^q}\)}-\frac4{h^q}\,\frac{q-1}{q^2}\irdg{\left|\nabla\(|u|^{q/2}\)\right|^2}\le0
\]
by~\eqref{GLSI} applied to $v=|u|^{q/2}$, if $t\mapsto q(t)$ solves the ordinary differential equation
\[
q'=2\,(q-1)\,.
\]
With $q(0)=p<2$, we obtain $q(t)=1+\(p-1\)e^{2t}$ and find that Nelson's time $t_*$ is determined by the condition $q(t_*)=2$ which means $e^{-2t_*}=p-1$. Replacing $t=t_*$ in~\eqref{Gross} completes the proof of~\eqref{GaussianInterpolation-improved}, which can be recast in the form of a stability result for~\eqref{GaussianInterpolation}.
%-------------------------------------------------------------------------
\begin{theorem}\label{Thm:StabSubcriticalGaussian} Let $d\ge1$ and $p\in[1,2)$. For any $f\in\mathrm H^1(\R^d,d\sigma)$ such that~\eqref{Cdt} holds,
\[
\nrmG{\nabla f}2^2-\frac1{2-p}\,\Big(\nrmG f2^2-\nrmG fp^2\Big)\ge\frac{2-p}2\nrmG{\nabla f}2^2\,.
\]\end{theorem}
%-------------------------------------------------------------------------

As a byproduct of the proof, with $t=t_*$ in~\eqref{Gross}, we have the mode-by-mode interpolation inequality
\[
\frac{\nrmG f2^2-\nrmG fp^2}{2-p}\le\sum_{k\ge1}\frac{1-(p-1)^k}{k\,(2-p)}\,\nrmG{\nabla f_k}2^2\quad\forall\,f\in\mathrm H^1(\R^d,d\sigma)\,,
\]
without imposing condition~\eqref{Cdt}, for any $p\in[1,2)$. For any $k\ge1$,
\[
\lim_{p\to2_-}\frac{1-(p-1)^k}{k\,(2-p)}=\lim_{p\to2_-}\frac{1-\big(1-(2-p)\big)^k}{k\,(2-p)}=1\,,
\]
so that no improvement should be expected by this method. This is very similar to the case of the critical exponent on the sphere of dimension $d\ge3$. In this sense $p=2$ is the critical case in the presence of a Gaussian weight, as \emph{all modes} are equally involved in the estimate of the constant. This is a limitation of the method which does not forbid a stability result for~\eqref{GLSI}, to be established by other methods.

\medskip Let us conclude this appendix with some bibliographic comments on the literature on inequality~\eqref{GaussianInterpolation}, for the Gaussian measure. The analogue of Proposition~\ref{Prop:GNSBE} in the Gaussian case is known from~\cite{MR2152502}; also see~\cite[Section~2.5]{doi:10.1142/S0218202518500574}). Assuming that not only condition~\eqref{Cdt} is satisfied, but also orthogonality conditions with all modes up to order $k_0\ge2$, then an improvement of the order of
\[
\frac{1-(p-1)^{k_0}}{k_0\,(2-p)}
\]
can be achieved for inequality~\eqref{GaussianInterpolation}, which is the counterpart of Theorem~\ref{StabSubcriticalSphereGNS} in the Gaussian case. This has been studied in~\cite{MR1796718} but we can refer to~\cite{MR2375056} for a more abstract setting and later papers, \emph{e.g.}, to~\cite{MR2446080,MR2127729} for results on compact manifolds and generalizations involving weights. For an overview of interpolation between Poincar\'e and logarithmic Sobolev inequalities from the point of view of Markov processes, and for some spectral considerations, we refer to~\cite[Chapter~6]{Wang:1250982}. Notice that hypercontractivity appears as one of the main motivations of the founding paper~\cite{bakry1985diffusions} of the \emph{carr\'e du champ} method.

%%%%%%%%%%%%%%%%%%%%%%%%%%%%%%%%%%%%%%%%%%%%%%%%%%%%%%%%%%%%%%%%%%%%%%%%%%
%%%%%%%%%%%%%%%%%%%%%%%%%%%%%%%%%%%%%%%%%%%%%%%%%%%%%%%%%%%%%%%%%%%%%%%%%%
\section{\emph{Carr\'e du champ} method and improved inequalities}\label{Appendix:BE-Sphere}

For sake of completeness, we collect various results of~\cite{MR3229793,DEKL,Dolbeault2017133,Dolbeault_2020b} and draw some new consequences. Computations similar to those of Section~\ref{Sec:Gamma2} can be found in~\cite{bidaut1991nonlinear} for the study of rigidity results in elliptic equations. For nonlinear parabolic flows, also see~\cite{Demange-PhD,MR2381156}. Other sections of this appendix collects results which are scattered in the literature, but additional details needed in Section~\ref{Sec:Improved-Carre} are given, for instance a sketch of the proof Proposition~\ref{Prop:ImprovedbyBE} or the computations in the case $p=2$.

%%%%%%%%%%%%%%%%%%%%%%%%%%%%%%%%%%%%%%%%%%%%%%%%%%%%%%%%%%%%%%%%%%%%%%%%%%
\subsection{Algebraic preliminaries}\label{Sec:Gamma2}

Let us denote the \emph{Hessian} by $\mathrm Hv$ and define the \emph{trace-free Hessian} by
\[
\mathrm Lv:=\mathrm Hv-\frac1d\,(\Delta v)\,g_d\,.
\]
We also consider the trace-free tensor
\[
\mathrm Mv:=\frac{\nabla v\otimes\nabla v}v-\frac1d\,\frac{|\nabla v|^2}v\,g_d\,,
\]
where $(\nabla v\otimes\nabla v)_{ij}:=\partial_iv\,\partial_jv$ and $\|\nabla v\otimes\nabla v\|^2=|\nabla v|^4=(g_d^{ij}\,\partial_iv\,\partial_jv)^2$ using Einstein's convention. Using
\[
\mathrm L:g_d=0\,,\quad\mathrm M:g_d=0\,,
\]
where $a:b$ denotes $a^{ij}\,b_{ij}$ and $\|a\|^2:=a:a$, and
\begin{align*}
&\|\mathrm Lv\|^2=\|\mathrm Hv\|^2-\frac1d\,(\Delta v)^2\,,\\
&\|\mathrm Mv\|^2=\left\|\frac{\nabla v\otimes\nabla v}v\right\|^2-\frac1d\,\frac{|\nabla v|^4}{v^2}=\frac{d-1}d\,\frac{|\nabla v|^4}{v^2}\,,
\end{align*}
we deduce from
\begin{align*}
\isd{\Delta v\,\frac{|\nabla v|^2}v}&=\isd{\frac{|\nabla v|^4}{v^2}}-2\isd{\mathrm Hv:\frac{\nabla v\otimes\nabla v}v}\\
&=\frac d{d-1}\isd{\|\mathrm Mv\|^2}-2\isd{\mathrm Lv:\frac{\nabla v\otimes\nabla v}v}-\frac2d\isd{\Delta v\,\frac{|\nabla v|^2}v}
\end{align*}
a first identity that reads
\be{Sphere:Firstd}
\isd{\Delta v\,\frac{|\nabla v|^2}v}=\frac d{d+2}\(\frac d{d-1}\isd{\|\mathrm Mv\|^2}-2\isd{\mathrm Lv:\frac{\nabla v\otimes\nabla v}v}\)\,.
\ee

The Bochner--Lichnerowicz--Weitzenb\"ock formula on $\S^d$ takes the simple form
\[
\frac12\,\Delta\,(|\nabla v|^2)=\|\mathrm Hv\|^2+\nabla(\Delta v)\cdot\nabla v+(d-1)\,|\nabla v|^2\,,
\]
where the last term, \emph{i.e.}, $\mathrm{Ric}(\nabla v,\nabla v)=(d-1)\,|\nabla v|^2$, accounts for the Ricci curvature tensor contracted with \hbox{$\nabla v \otimes\nabla v$}. An integration of this formula on $\S^d$ shows a second identity,
\be{Sphere:Secondd}
\isd{(\Delta v)^2}=\frac d{d-1}\isd{\|\mathrm Lv\|^2}+d\isd{|\nabla v|^2}\,.
\ee
Hence
\begin{align*}
\mathscr K[v]:=&\isd{\(\Delta v+\kappa\,\frac{|\nabla v|^2}v\)\(\Delta v+(\beta-1)\,\frac{|\nabla v|^2}v\)}\\
=&\isd{(\Delta v)^2}+(\kappa+\beta-1)\isd{\Delta v\,\frac{|\nabla v|^2}v}+\kappa\,(\beta-1)\isd{\frac{|\nabla v|^4}{v^2}}
\end{align*}
can be rewritten using~\eqref{Sphere:Firstd} and~\eqref{Sphere:Secondd} as
\begin{align*}
\mathscr K[v] &=\frac d{d-1}\isd{\|\mathrm Lv\|^2}+ d\isd{|\nabla v|^2}\\
&\hspace*{1cm}+(\kappa+\beta-1)\,\frac d{d+2}\(\frac d{d-1}\isd{\|\mathrm Mv\|^2}-2\isd{\mathrm Lv:\mathrm Mv}\)\\
&\hspace*{2cm}+\kappa\,(\beta-1)\,\frac d{d-1}\isd {\|\mathrm Mv\|^2}\\
&=\frac d{d-1}\isd{\Big(\|\mathrm Lv\|^2-2\,b\,\mathrm Lv:\mathrm Mv+c\,\|\mathrm Mv\|^2\Big)}+d\isd{|\nabla v|^2}\\
&=\frac d{d-1}\isd{\Big(\|\mathrm Lv-\,b\,\mathrm Mv\|^2+\(c-b^2\)\|\mathrm Mv\|^2\Big)}+d\isd{|\nabla v|^2}\\
&=\frac d{d-1}\isd{\|\mathrm Lv-\,b\,\mathrm Mv\|^2}+\(c-b^2\)\isd{\frac{|\nabla v|^4}{v^2}}+d\isd{|\nabla v|^2}
\end{align*}
where
\[
b=(\kappa+\beta-1)\,\frac{d-1}{d+2}\quad\mbox{and}\quad c=\frac d{d+2}\,(\kappa+\beta-1)+\kappa\,(\beta-1)\,.
\]

Let $\kappa=\beta\,(p-2)+1$. The condition $\gamma:=c-b^2\ge0$ amounts to
\be{gamma:beta}
\gamma=\frac d{d+2}\,\beta\,(p-1)+\big(1+\beta\,(p-2)\big)\,(\beta-1)-\(\frac{d-1}{d+2}\,\beta\,(p-1)\)^2\,,
\ee
where $\gamma=-\,\big(A\,\beta^2-2\,B\,\beta+C\big)$ with
\[
A=\(\frac{d-1}{d+2}\,(p-1)\)^2+2-p\,,\quad B=\frac{d+3-p}{d+2}\quad\mbox{and}\quad C=1\,.
\]
A necessary and sufficient condition for the existence of a $\beta$ such that $\gamma\ge0$ is that the reduced discriminant is nonnegative, which amounts to
\[
B^2-A\,C=\frac{4\,d\,(d-2)}{(d+2)^2}\,\big(p-1\big)\(2^*-p\)\ge0\,.
\]
Summarizing, we have the following result, which can be found in~\cite{MR3229793} for a general manifold with positive Ricci curvature.
%-------------------------------------------------------------------------
\begin{lemma}\label{Lem:BE:algebraic} With the above notation, for any smooth function $v$ on $\S^d$, we have
\[
\mathscr K[v]\ge\gamma\isd{\frac{|\nabla v|^4}{v^2}}+d\isd{|\nabla v|^2}
\]
for some $\gamma>0$ given in terms of $\beta$ by~\eqref{gamma:beta} if $p\in(1,2^*)$.
\end{lemma}
%-------------------------------------------------------------------------
Notice that we recover the expression for $\gamma$ in~\eqref{gamma1} if we take $\beta=1$. The case $p=2$ does not add any difficulty compared to $p\neq2$.

%%%%%%%%%%%%%%%%%%%%%%%%%%%%%%%%%%%%%%%%%%%%%%%%%%%%%%%%%%%%%%%%%%%%%%%%%%
\subsection{Diffusion flow and monotonicity}\label{Sec:NLFlow}

Assume that $u$ is a positive solution of
\be{FDE}
\frac{\partial u}{\partial t}=u^{-p\,(1-m)}\(\Delta u+(m\,p-1)\,\frac{|\nabla u|^2}u\)\,.
\ee
In the linear case $m=1$, $u^p$ solves the heat equation. Otherwise we deal with the nonlinear case either of a fast diffusion flow with $m<1$ or of a solution of the porous media equation with $m>1$. We claim that
\[
\frac d{dt}\nrms up^2=0\quad\mbox{and}\quad\frac d{dt}\nrms u2^2=2\,(p-2)\isd{u^{-\,p\,(1-m)}\,|\nabla u|^2}\,.
\]
Let us assume that the parameters $\beta$ and $m$ are related by
\be{Id:mbeta}
m=1+\frac2p\(\frac1\beta-1\)\,.
\ee
If $v$ is a function such that $u=v^\beta$, then $v$ solves
\[
\frac{\partial v}{\partial t}=v^{2-2\,\beta}\(\Delta v+\kappa\,\frac{|\nabla v|^2}v\)\,,
\]
with $\kappa=\beta\,(p-2)+1$ and as a consequence we find that
\[
\frac d{dt}\nrms u2^2=2\,(p-2)\,\beta^2\isd{|\nabla v|^2}\,.
\]
Similarly, we find that
\be{BEcomputation}
\frac d{dt}\nrms{\nabla u}2^2=-\,2\isd{\(\beta\,v^{\beta-1}\,\frac{\partial v}{\partial t}\)\(\Delta v^\beta\)}=-\,2\,\beta^2\,\mathscr K[v]\,.
\ee
By eliminating $\beta$ in~\eqref{gamma:beta} using~\eqref{Id:mbeta}, we obtain
\be{gamma:m}
\gamma=\frac{\gamma_0+\gamma_1\,d+\gamma_2\,d^2}{(d+2)^2\,\big(2-p\,(1-m)\big)^2}\,,
\ee
with $\gamma_0=4\,(m\,p-1)^2$, $\gamma_1=-\,4\,p\,\big(\!m-3+p\,(2-m)\,(1+m)\!\big)$, \hbox{$\gamma_2=(m^2-2\,m+5)\,p^2-12\,p+8$}. The condition $\gamma\ge0$ determines the range $m_-(d,p)\le m\le m_+(d,p)$ of \emph{admissible} parameters $m$, where $m_\pm(d,p)$ is given by~\eqref{admissible.m}. Summarizing, we have the following result (also see~\cite{MR3229793}).
%-------------------------------------------------------------------------
\begin{lemma}\label{Lem:BE:Flow} Assume that $p\in(1,2^*)$ and $m\in[m_-(d,p),m_+(d,p)]$. If $u$ solves~\eqref{FDE}, then we have
\be{FDE:decay}
\frac1{2\,\beta^2}\,\frac d{dt}\(\nrms{\nabla u}2^2-d\,\mathcal E_p[u]\)\le-\,\gamma\isd{\frac{|\nabla v|^4}{v^2}}\,,
\ee
where $v=u^{1/\beta}$ with $\beta$ and $\gamma$ given in terms of $m$ by~\eqref{gamma:beta} and~\eqref{gamma:m} respectively.
\end{lemma}
%-------------------------------------------------------------------------
Notice that the case of the linear flow corresponds to the case $m=\beta=1$ and $v=u$.
\begin{proof}[Proof of Lemma~\ref{Lem:BE:Flow}] For a smooth solution, the result follows from~\eqref{BEcomputation} and Lemma~\ref{Lem:BE:algebraic}. The result for a general solution is obtained by standard regularization procedures. \end{proof}

%%%%%%%%%%%%%%%%%%%%%%%%%%%%%%%%%%%%%%%%%%%%%%%%%%%%%%%%%%%%%%%%%%%%%%%%%%
\subsection{Interpolation}\label{Sec:Interpolation}

Depending on the value of $p$, we shall consider various interpolation inequalities. Let us define
\be{delta}
\delta:=\frac{p-\,(4-p)\,\beta}{2\,\beta\,(p-2)}\quad\mbox{if}\quad p>2\,,\quad\delta:=1\quad\mbox{if}\quad p\in[1,2]\,.
\ee
%-------------------------------------------------------------------------
\begin{lemma}\label{Lem:BE:Interpolation} If one of the conditions
\begin{enumerate}
\item[(i)] $p\in(1,2^\#)$ and $\beta=1$ (so that $\delta=1$),
\item[(ii)] $p\in(2,2^*)$, $\beta>1$, and $\beta\le2/(4-p)$ if $p<4$, 
\end{enumerate}
is satisfied, then $u=v^\beta$ is such that
\be{Interp:Demange}
\isd{\frac{|\nabla v|^4}{|v|^2}}\ge\frac1{\beta^2}\,\frac{\isd{|\nabla u|^2}\isd{|\nabla v|^2}}{\(\isd{|u|^2}\)^\delta\,\(\isd{|u|^p}\)^\frac{\beta-1}{\beta\,(p-2)}}\,.
\ee
\end{lemma}
%-------------------------------------------------------------------------
\noindent Case (ii) was originally proved in~\cite{Demange-PhD,MR2381156} and we refer to~\cite{DEKL} for a proof in the case of the ultraspherical operator.
\begin{proof}[Proof of Lemma~\ref{Lem:BE:Interpolation}] In case (i), $v=u$ and inequality~\eqref{Interp:Demange} is a consequence of the Cau\-chy--Schwarz inequality
\[
\isd{|\nabla v|^2}=\isd{\frac{|\nabla v|^2}v\cdot v}\le\(\isd{\frac{|\nabla v|^4}{v^2}}\)^\frac12\(\isd{|u|^2}\)^\frac12\,,
\]
Cases (i) and (ii) follow from two H\"older inequalities.
\\[4pt]
(1) With $\frac12+\frac{\beta-1}{2\,\beta}+\frac1{2\,\beta}=1$, we deduce from
\[
\isd{|\nabla v|^2}=\isd{\frac{|\nabla v|^2}v\cdot 1\cdot v}\le\(\isd{\frac{|\nabla v|^4}{v^2}}\)^\frac12\(\isd 1\)^\frac{\beta-1}{2\,\beta}\(\isd{|u|^2}\)^\frac1{2\,\beta}\,,
\]
and the assumption that $d\mu$ is a probability measure, the first estimate
\[
\(\isd{\frac{|\nabla v|^4}{v^2}}\)^\frac12\ge\frac{\isd{|\nabla v|^2}}{\(\isd{|u|^2}\)^\frac1{2\,\beta}}\,.
\]
(2) With $\frac12+\frac{\beta-1}{\beta\,(p-2)}+\frac{2-(4-p)\,\beta}{2\,\beta\,(p-2)}=1$ and $\delta_0=\frac{2-\,(4-p)\,\beta}{2\,\beta\,(p-2)}$, H\"older's inequality shows that
\begin{multline*}
\frac1{\beta^2}\isd{|\nabla u|^2}=\isd{v^{2(\beta-1)}\,|\nabla v|^2}=\isd{\frac{|\nabla v|^2}v\cdot v^\frac{p\,(\beta-1)}{p-2}\cdot v^{2\,\beta\,\delta_0}}\\
\le\(\isd{\frac{|\nabla v|^4}{v^2}}\)^\frac12\(\isd{|u|^p}\)^\frac{\beta-1}{\beta\,(p-2)}\(\isd{|u|^2}\)^{\delta_0}\,,
\end{multline*}
from which we deduce the second estimate
\[
\(\isd{\frac{|\nabla v|^4}{v^2}}\)^\frac12\ge\frac1{\beta^2}\,\frac{\isd{|\nabla u|^2}}{\(\isd{|u|^2}\)^{\delta_0}\,\(\isd{|u|^p}\)^\frac{\beta-1}{\beta\,(p-2)}}\,.
\]
The combination of our two estimates proves~\eqref{Interp:Demange} with $\delta=\delta_0+1/(2\,\beta)$.
\end{proof}

Using~\eqref{Id:mbeta}, condition (ii) in Lemma~\ref{Lem:BE:Interpolation} is changed into the condition that $2/p\le m<1$ and we may notice as in~\cite{Demange-PhD,MR2381156} that it is always satisfied if we choose $\beta=4/(6-p)$ corresponding to an \emph{admissible} fast diffusion exponent $m=(p+2)/(2\,p)$, for any $p\in(2,2^*)$. By ``admissible'', one should understand $m_-(d,p)\le m\le m_+(d,p)$, so that $\gamma$ is nonnegative. With the choice of $m=(p+2)/(2\,p)$, we find $\delta=1-p/8$.

%%%%%%%%%%%%%%%%%%%%%%%%%%%%%%%%%%%%%%%%%%%%%%%%%%%%%%%%%%%%%%%%%%%%%%%%%%
\subsection{Improved functional inequalities}\label{Appendix:monotonicity}

Let us denote the \emph{entropy} and the \emph{Fisher information} respectively by
\[
\mathsf e:=\frac 1{p-2}\(\nrms up^2-\nrms u2^2\)\quad\mbox{and}\quad\mathsf i:=\nrms{\nabla u}2^2\,,
\]
and let $\gamma$ and $\delta$ be given respectively by~\eqref{gamma:beta} and~\eqref{delta}. Up to the replacement of $u$ by $u/\nrms up$, with no loss of generality, we shall assume that
\[
\nrms up=1\,.
\]
We learn from~\eqref{FDE:decay} and~\eqref{Interp:Demange} that
\be{Ineq:EDO}
\(\mathsf i-\,d\,\mathsf e\)'\le\frac{\gamma\,\mathsf i\,\mathsf e'}{\beta^2\,\big(1-(p-2)\,\mathsf e\big)^\delta}\,.
\ee
Solving the ordinary differential equation in the equality case of~\eqref{Ineq:EDO} is equivalent to solving 
\[
\frac d{dt}\big(\mathsf i-\,d\,\varphi(\mathsf e)\big)=\frac\gamma{\beta^2}\,\frac{\mathsf e'}{\big(1-(p-2)\,\mathsf e\big)^\delta}\(\mathsf i-\,d\,\varphi(\mathsf e)\)\,,
\]
where $\varphi$ solves
\be{ODE}
\varphi'(s)=1+\frac\gamma{\beta^2}\,\frac{\varphi(s)}{\(1\,-\,(p-2)\,s\)^\delta}\,.
\ee
The reader is invited to check that the solution of~\eqref{ODE} with initial datum $\varphi(0)=0$ is given by~\eqref{phifunction} if $m=1$ and by~\eqref{phifunction:mp} with $\zeta=2\,\gamma/\big(\beta\,(1-\beta)\big)$ in the nonlinear case. We learn from~\eqref{Ineq:EDO} that
\[
\(\mathsf i-\,d\,\varphi(\mathsf e)\)'\le\frac\gamma{\beta^2}\,\frac{\mathsf e'}{\big(1-(p-2)\,\mathsf e\big)^\delta}\(\mathsf i-\,d\,\varphi(\mathsf e)\)\,.
\]
This is enough to prove the following result.
%-------------------------------------------------------------------------
\begin{proposition}\label{Prop:ImprovedbyBE} With the above notation, we claim that
\[
\mathsf i\ge\,d\,\varphi(\mathsf e)\,.
\]
\end{proposition}
%-------------------------------------------------------------------------
\begin{proof} Let us give the scheme of a proof. Let $\tilde\gamma:=\gamma/\beta^2$, in order to simplify notation. We can argue as follows:
\begin{enumerate}
\item $\mathsf i'+2\,d\,\mathsf i=\(\mathsf i-d\,\mathsf e\)'\le0$ shows that
\[
0\le\mathsf i(t)\le\mathsf i(0)\,e^{-2\,d\,t}
\]
and in particular \hbox{$\displaystyle\lim_{t\to+\infty}\mathsf i(t)=0$}.
\item As $t\to+\infty$, $\mathsf e$ converges to a constant, hence \hbox{$\displaystyle\lim_{t\to+\infty}\mathsf e(t)=0$}.
\item From~\eqref{Ineq:EDO}, we learn that
\[
\(\mathsf i-d\,\mathsf e\)'\le d\,\tilde\gamma\,\mathsf e\,\mathsf e'=\frac12\,d\,\tilde\gamma\,(\mathsf e^2)'\,,
\]
where the inequality follows from $1-(p-2)\,\mathsf e\le1$ and $\mathsf i\ge d\,\mathsf e$.
\item It follows from $\(\mathsf i-d\,\mathsf e\)'\le0$ that $\mathsf i\ge d\,\mathsf e$ using an integration from any $t\ge0$ to~$+\infty$.
\item Unless $u$ is a constant, we read from $\(\mathsf i-d\,\mathsf e\)'\le\frac12\,\tilde\gamma\,d\,(\mathsf e^2)'$ that $\mathsf i-d\,\mathsf e>\frac12\,\tilde\gamma\,d\,\mathsf e^2$, using again an integration from any $t\ge0$ to~$+\infty$.
\item Take some $\vartheta\in(0,1)$ and consider the solution of
\be{overlinevarphi}
\overline\varphi'(s)=1+\frac{\vartheta\,\tilde\gamma\,\overline\varphi(s)}{\(1\,-\,(p-2)\,s\)^\delta}\,,\quad\overline\varphi(0)=0\,.
\ee
In the spirit of~\eqref{Ineq:EDO}, we have a following chain of elementary estimates:
\[
\(\mathsf i-\,d\,\vartheta\,\overline\varphi(\mathsf e)\)'\le\(\mathsf i-\,d\,\overline\varphi(\mathsf e)\)'+d\,(1-\vartheta)\,\(\overline\varphi(\mathsf e)\)'\le\(\mathsf i-\,d\,\overline\varphi(\mathsf e)\)'
\]
and obtain
\be{varphi''}
\(\mathsf i-\,d\,\vartheta\,\overline\varphi(\mathsf e)\)'\le\frac{\tilde\gamma\,\mathsf e'}{\big(1-(p-2)\,\mathsf e\big)^\delta}\(\mathsf i-\,d\,\vartheta\,\overline\varphi(\mathsf e)\)\,.
\ee
We know that $\overline\varphi(0)=0$ and read from~\eqref{overlinevarphi} that $\overline\varphi'(0)=1$ and
\[
\overline\varphi''(0)=\vartheta\,\tilde\gamma\,\overline\varphi'(0)=\vartheta\,\tilde\gamma\,,
\]
so that $\overline\varphi(\mathsf e)-\mathsf e\sim\frac12\,\vartheta\,\tilde\gamma\,\mathsf e^2$ as $\mathsf e\to0$. Using $\mathsf i-d\,\mathsf e>\frac12\,\tilde\gamma\,d\,\mathsf e^2$, we learn that
\[
\mathsf i-\,d\,\overline\varphi(\mathsf e)\ge\frac12\,\tilde\gamma\,d\,(1-\vartheta)\,\mathsf e^2\,\big(1+O(\mathsf e)\big)
\]
 for $\mathsf e=\mathsf e(t)$ small enough, \emph{i.e}, for $t>0$ large enough.
\item It is simple to check from~\eqref{varphi''} that $\mathsf i-\,d\,\vartheta\,\overline\varphi(\mathsf e)$ cannot change sign.
\item We conclude as above that $\mathsf i-\,d\,\vartheta\,\overline\varphi(\mathsf e)\ge0$ using an integration from any $t\ge0$ to~$+\infty$.
\item Finally, we consider the limit as $\vartheta\to1_-$.
\end{enumerate}
Altogether, we conclude that $\mathsf i\ge\,d\,\varphi(\mathsf e)$, where $\varphi$ solves~\eqref{ODE}. This completes the scheme of the proof of Proposition~\ref{Prop:ImprovedbyBE}.
\end{proof}

%%%%%%%%%%%%%%%%%%%%%%%%%%%%%%%%%%%%%%%%%%%%%%%%%%%%%%%%%%%%%%%%%%%%%%%%%%
%%%%%%%%%%%%%%%%%%%%%%%%%%%%%%%%%%%%%%%%%%%%%%%%%%%%%%%%%%%%%%%%%%%%%%%%%%
\section*{Acknowledgements}
This work has been partially supported by the Project EFI (ANR-17-CE40-0030) of the French National Research Agency (ANR). G.B.~has been funded by the European Union’s Horizon 2020 research and innovation program under the Marie Sk\l odow\-ska-Curie grant agreement No.~754362.\\
\noindent{\scriptsize\copyright\,2023 by the authors. This paper may be reproduced, in its entirety, for non-commercial purposes.}
\newpage
%%%%%%%%%%%%%%%%%%%%%%%%%%%%%%%%%%%%%%%%%%%%%%%%%%%%%%%%%%%%%%%%%%%%%%%%%%
%%%%%%%%%%%%%%%%%%%%%%%%%%%%%%%%%%%%%%%%%%%%%%%%%%%%%%%%%%%%%%%%%%%%%%%%%%
%\bibliographystyle{siam}\small\bibliography{References}

%\bigskip\begin{center}\rule{2cm}{0.5pt}\end{center}\bigskip
\end{document}